\documentclass{amsart}
\usepackage[paper]{hzmath}

\makeatletter
\@namedef{subjclassname@2020}{%
  \textup{2020} Mathematics Subject Classification}
\makeatother

\begin{document}

\title[uniform propagation for CBO]{Uniform-in-time weak propagation of chaos for consensus-based optimization}

\author[Bayraktar]{Erhan Bayraktar}
\address{%
	Department of Mathematics,
	University of Michigan,
	Ann Arbor, MI 48109.}
\email{erhan@umich.edu  }

\author[Ekren]{Ibrahim Ekren}
\address{%
	Department of Mathematics,
	University of Michigan,
	Ann Arbor, MI 48109.}
\email{iekren@umich.edu  }

\author[Zhou]{Hongyi Zhou}
\address{%
	Department of Mathematics,
	University of Michigan,
	Ann Arbor, MI 48109.}
\email{hongyizh@umich.edu	}

\date{February 1, 2025}

\begin{abstract}
	We study the uniform-in-time weak propagation of chaos for the consensus-based optimization (CBO) method on a bounded searching domain.
	We apply the methodology for studying long-time behaviors of interacting particle systems developed in the work of Delarue and Tse~\cite{DelarueTse2021}. 
	Our work shows that the weak error has order $O(N^{-1})$ uniformly in time, where $N$ denotes the number of particles. 
	The main strategy behind the proofs are the decomposition of the weak errors using the linearized Fokker-Planck equations and the exponential decay of their Sobolev norms.
	Consequently, our result leads to the joint convergence of the empirical distribution of the CBO particle system to the Dirac-delta distribution at the global minimizer in population size and running time in Wasserstein-type metrics.
\end{abstract}

\maketitle

\tableofcontents

\section{Introduction}

The \emph{consensus-based optimization} (CBO) is a gradient-free optimization algorithm inspired by Laplace's principle
\begin{equation*}
	\lim_{\alpha \to +\infty} \left( -\frac{1}{\alpha} \log \int_{\R^d} e^{-\alpha \efn(x)} \mu(dx) \right) = \inf_{x \in \supp(\mu)} \efn(x) \,,
\end{equation*}
which holds for any absolutely continuous probability measure $\mu$ on $\R^d$, where $\efn: \R^d \to \R$ is the objective function to be minimized.
Restricting the searching scope to a bounded domain, the CBO method runs a dynamic system of $N$ interacting particles, denoted by $(X^{N,i})_{i=1,\dots,N}$, described by the system of stochastic differential equations
\begin{subequations}
\label{e:CBO-particles}
    \begin{equation}
        d X^{N,i}_t = - \lambda (X^{N,i}_t - m^N_t) dt + \sigma \abs{X^{N,i}_t - m^N_t} \phi(X^{N,i}_t) dW^i_t
    \end{equation}
	for $i=1,\dots,N$, where $W^i$ are i.i.d. standard $d$-dimensional Brownian motions, $\phi$ is a compactly supported smooth function, and 
    \begin{equation}
        m^N_t \defeq \frac{\sum_{j=1}^N X^{N,j}_t \exp(-\alpha \cE(X^{N,j}_t))}{\sum_{j=1}^N \exp(-\alpha \cE(X^{N,j}_t))} \,, \qquad \nu^N_t \defeq \frac{1}{N} \sum_{j=1}^N \delta_{X^{N,j}_t} \,.
    \end{equation}
\end{subequations}
Previous works on the CBO methods have shown that the particles can run arbitrarily close to the unique global minimizer $x^\ast$ (if it exists), for large enough $t$, $N$, and $\alpha$, whereas the choices of those quantities may depend on each other.

In this research, we study the proximity of the above system to its mean-field limit process over the infinite time horizon.
We will demonstrate that, under appropriate choices of the constant parameters $\lambda, \sigma$, and $\alpha$, the empirical distribution $(\nu^N_t)_{t \ge 0}$ of the system~\eqref{e:CBO-particles} converges to a limit $(\bar \nu_t)_{t \ge 0}$ uniformly in time, in the sense that 
\begin{equation*}
	\sup_{t \ge 0} \abs{\Phi(\bar\nu_t) - \E \Phi(\nu^N_t)} \le O(N^{-1})
\end{equation*}
for a family of functions $\Phi: \prob(\R^d) \to \R$.
This admits independent choices of the number $N$ of particles and the running time $t$ when given an error tolerance in practice. 
In the end, our work leads to the convergence of $\nu^N_t$ to some limit point $\tilde x$ as $N, t \to \infty$ independently, in the \emph{centered Fourier-Wasserstein distance} and the \emph{centered Wasserstein distance}.
Our result admits independent choices of $t$ and $N$ subject to a given error tolerance (so the system size $N$ need not grow when the running time $t$ gets large).

\subsection{Background}

Optimization has long been a significant concern in every aspect of science.
Traditional optimization algorithms, such as the gradient descent method, rely heavily on the convex structure of the objective function and require computations of derivatives.
While optimization problems in reality show high non-convexity, there is a rapid rise in the field of gradient-free algorithms.

In recent decades, several searching-based algorithms have been introduced, including Simulated Annealing~\cite{AartsKorst1991}, evolutionary algorithms~\cite{BaeckFogelMichalewicz1997,Fogel2006}, genetic algorithms~\cite{Holland1992,ReevesRowe2002}, 
and most notably, Particle Swarm Optimization (PSO) introduced in the 1990s~\cite{KennedyEberhart1995,ShiEberhart1998}. 
In the PSO method, multiple agents in the system explore the state space according to their currently best-known positions in the state space.
The information-sharing and best-position searching mechanism in PSO later inspired Consensus-Based Optimization (CBO). 

Formally established in~\cite{PinnauTotzeckTseMartin2017}, CBO is an innovative optimization algorithm that uses opinion dynamics of multi-agent systems. 
It relies solely on the values of the objective/cost function and avoids the heavy computations of gradients, which is notably helpful for problems on non-smooth, non-convex, and even noisy landscapes.
Technically, the CBO method operates an interacting particle system $\{X^{N,i}\}_{i \in [N]}$, where particles (agents) share their performances with each other to form a \emph{consensus} of the (regularized) ``current best'' at time $t$ defined by  
\begin{equation*}
	m^N_t \defeq \frac{\sum_{j=1}^N X^{N,j}_t \exp(-\alpha \cE(X^{N,j}_t))}{\sum_{j=1}^N \exp(-\alpha \cE(X^{N,j}_t))} \,.
\end{equation*}
The particles follow the dynamics
\begin{equation*}
    d X^{N,i}_t = - \lambda (X^{N,i}_t - m^N_t) dt + \sigma \abs{X^{N,i}_t - m^N_t} \phi(X^{N,i}_t) dW^i_t \,,
\end{equation*}
which drive the particles towards the current best in infinitesimally small steps.

The CBO method displays high performance in practice, even for high-dimensional optimization problems, because of its gradient-free essence. 
With additional controls in the particle dynamics, it is applicable to solve constrained optimization problems~\cite{CarrilloJinZhangZhu2024}.
The work of Huang, Qiu, and Riedl~\cite{HuangQiuRiedl2024} presents a consensus-based algorithm searching for Nash equilibria (or saddle point) in minimax games, wherein two groups of agents actively exchange information and weight the observed performances to conditionally optimize their choices, approaching an equilibrium. 
Such methods do not require the convexity-concavity of objective functions and also avoid computing gradients.

\subsection{Main problem and our contributions}

Similar to many other interacting particle systems, the large population (mean-field) limit of the CBO system, given by 
\begin{subequations}
    \label{e:cbo-mf-intro}
    \begin{equation}
        d \bar X_t = -\lambda (\bar X_t - \bar m_t) dt + \sigma \abs{\bar X_t - \bar m_t} d \bar W_t \,, 
    \end{equation}
    where 
    \begin{equation}
        \bar m_t \defeq \frac{\int x e^{-\alpha \efn(x)} \bar \nu_t(dx)}{\int e^{-\alpha \efn(x)} \bar \nu_t(dx)} \,, \qquad \bar \nu_t = \law(\bar X_t) \,,
    \end{equation}
\end{subequations}
has drawn high interest among researchers. 
Properties of the limit process are studied in the work of Carrilio et al. (see~\cite{CarriloChoiTotzeckTse2018} for instance) and also~\cite{FornasierKlockRiedl2021}.
Those works show an exponential convergence rate of $\bar X_t$ to some limit point, and the practical efficacy of the algorithm depends on the quantitative proximity of the particle system~\eqref{e:CBO-particles} to its mean-field limit~\eqref{e:cbo-mf-intro}, which is called the \emph{propagation of chaos}.
The strong propagation of chaos (measured in $L^2$-norm) over finite time horizons has been established in~\cite{HuangQiu2022,GerberHoffmannVaes2023} under the framework of~\cite{Sznitman1991}.
Those results rely strictly on the finiteness of the time horizon, which causes a trade-off against the convergence of the mean-field process in a long time.

However, as demonstrated in~\cite{CarriloChoiTotzeckTse2018}, the CBO mean-field limit process displays exponential convergence in infinite time, subject to some regularity conditions. 
This raises the question of whether the propagation of chaos still holds when the algorithm runs for an arbitrarily long time. 
Denoting by $(\nu^N_t)_{t \ge 0}$ the empirical distributions of the CBO particle system over the infinite time horizon, we would like to show that it is close to the weak limit $(\bar \nu_t)_{t \ge 0}$ in the sense that 
\begin{equation*}
	\sup_{t \ge 0} \abs{ \E[\Phi(\nu^N_t)] - \Phi(\bar \nu_t) } \le \frac{C}{N}
\end{equation*}
for some constant $C$, where $\Phi$ is some smooth enough function.
One nonlinear example of concern is the centered Fourier-Wasserstein distance
\begin{equation*}
    \norm{\nu^N_t \circ \tau_{\ip{\id, \nu^N_t}} - \rho_\refer }_{-s,2} 
\end{equation*}
to some reference measure $\rho_{\text{ref}}$ with mean 0. 
Here the Sobolev $(-s,2)$-norm is defined by 
\begin{equation*}
    \norm{q}_{-s,2}^2 \defeq \int_{\R^d} (1+\abs{\xi}^2)^{-s} \abs{\int_{\R^d} e^{-i2\pi \xi^\top x} q(dx) }^2 d\xi \,.
\end{equation*}

Such a result belongs to a specific category, the \emph{uniform-in-time propagation of chaos}.
In fact, it is more challenging because of the commonly seen accumulation of errors in large time intervals.
The propagation of chaos uniformly in time thus relies on some sense of ergodicity of the mean-field limit process. 
Some examples are studied in~\cite{Malrieu2001,Malrieu2003,GuillinLebrisMonmarche2023,RosenzweigSerfaty2023,Lacker2021,LackerFlem2023}, and more details are discussed in~\cite{DelarueTse2021}.
In fact, one notices that the CBO particle system~\eqref{e:CBO-particles} stops evolving when all particles are at the same point. 
Along with the results in~\cite{CarriloChoiTotzeckTse2018,FornasierKlockRiedl2021}, it can be seen that for any initial distribution $\mu$, there is some point $\tilde x^\mu$ such that the Dirac-delta measures $\delta_{\tilde x^\mu}$ is the attractive invariant measure, which is the key to the global time estimates.

In this paper, we show that the methodology in~\cite{DelarueTse2021} applies to the study of the uniform-in-time propagation of chaos for the consensus-based optimization algorithm. 
Our main contribution is an upper bound of the form 
\begin{equation*}
	\sup_{t \ge 0} \abs{\E[\Phi(\nu^N_t)] - \Phi(\bar \nu_t)} \le \frac CN \,,
\end{equation*}
where the constant $C$ has no dependence on time.  
As a consequence, it gives an estimate of the form 
\begin{equation*}
	\abs{\E[\Phi(\nu^N_t)] - \Phi(\delta_{x^\ast})} \le C (N^{-1} + e^{-\kappa t} + \ep_0) \,, \qquad \forall N \in \Z_+, \; t \ge 0 \,,
\end{equation*}
where $\ep_0$ is some intrinsic error of the CBO method due to the finite-ness of $\alpha$ (see Section~4 in~\cite{CarriloChoiTotzeckTse2018} for detail).
Besides the existence of $\ep_0$, this is a substantial improvement from the previous works since it establishes a quantitative proximity between the particle system and its mean-field limit independent of the time horizon.
That provides users with a clear reference to balance between computational accuracy and computing resources. 
Users may choose the number $N$ of particles to run without consideration of the running time $t$, while the strong propagation-of-chaos results of~\cite{HuangQiu2022,GerberHoffmannVaes2023} requires a choice of $N$ depending on $t$ when subject to the same error tolerance. 
It shows that the CBO algorithm gives an acceptable approximation to the optimizer even on occasions where computing resources are scarce.

To demonstrate our argument, for the first step, we decompose the temporal weak error $\E[\Phi(\nu^N_t)] - \Phi(\bar \nu_t)$ into integrals over the solution $\ufn$ to the master equation of the CBO process, following the ideas of~\cite{DelarueTse2021,Tse2021,ChassagneuxSzpruchTse2022}.
The integrands therein can be further expressed as linear operations on solutions to the linearized Fokker-Planck equation defined in~\cite{Cormier2024,MischlerMouhotWennberg2015}.
In particular, the main objective is to show that the second derivatives 
\begin{equation}
	\label{e:2nd-deri-u}
	\p_{(y_2)_i} \p_{(y_1)_i} \frac{\delta^2 \ufn}{\delta m^2} (t,\mu,y_1,y_2)
\end{equation}
has exponential decay in the time variable $t$, which relies on the ergodicity of the CBO Fokker-Planck equations. 

We adopt the methodology of~\cite{DelarueTse2021} to split the norm of some linearized Fokker-Planck equation (to be defined in Section~\ref{s:main-result}) at an attracting invariant measure $\delta_{\tilde x}$,
\begin{equation*}
    \p_t q_t = (\lin^\ast + \a^\ast) q_t + r_t \,,
\end{equation*}
using its adjoint backward Cauchy equation of the form 
\begin{equation*}
    \p_t w(t) + (\lin + \a) w(t) = 0 \,, \quad t \in [0,T] \,.
\end{equation*} 
It is worth noticing that, under the setting of the consensus-based optimization, this equation does not directly exhibit a backward-in-time exponential decay as in Lemma 3.11 of~\cite{DelarueTse2021}.
The magnitude of $w$ has a strict lower bound of order $O(\norm{w(T)})$. 
This is because the second-order term in $\lin$, which comes from the volatility of~\eqref{e:CBO-particles}, is not bounded from below, and the operator $\lin$ is thus not fully diffusive. 

Instead, the derivatives of $w$ satisfy some equation of the form 
\begin{equation*}
    \p_t \p_{x_i} w(t) + (\lin^1 + \a^1 - \lambda) \p_{x_i} w(t) = 0 \,, \quad t \in [0,T] \,.
\end{equation*}
Thus, using the Feynman-Kac formula, we see that the derivatives of $w$ admit some exponential decay up to normalization, in the sense that 
\begin{equation*}
	\abs{\p_{x_i} w(t) - \varphi^{\rem} \cdot \grad_x  
    w(T,\tilde x)} \le C e^{-\lambda (T-t)}  
\end{equation*}
for some $\varphi^\rem \in \R^d$. 
This is also seen by a sequence of analyses on geometric Brownian motions (GBM) with non-constant volatilities, 
where both the GBMs themselves and their derivatives with respect to the initial positions display exponential decay.
Those results give us an analog of Proposition 4.8 of \cite{DelarueTse2021}, for (only) the derivatives of $w$, and consequently the decay of the linearized Fokker-Planck equation, in the form 
\begin{equation*}
    \norm{q_T - q_\infty \cdot \grad \delta_{\tilde x}}_{(n,\infty)'} \le C e^{-\kappa_0 T} \,.
\end{equation*}
Finally, we follow the idea of Section~4 in \cite{DelarueTse2021} to demonstrate our main arguments.

One concession we have made throughout the problem is the cutoff on the consensus-based optimization algorithm.
When the users are confident about the possible range of the global minimizer, it is reasonable to research only within the scope while discarding the rest.
The reason for such concession is that, without the cutoff function, the estimates on the quantity~\eqref{e:2nd-deri-u} depend exponentially on the moments of the measure argument $\mu$,
which then leads to the demand for a uniform-in-time upper bound on the quantity 
\begin{equation*}
	\E \left[ \exp \left( \frac{C}{N} \sum_{j=1}^N \abs{X^{N,j}_t}^2 \right) \right] \,.
\end{equation*}
This need not be finite when $C$ is very large. 
To avoid such occasions, we add the cutoff function $\phi$ to the volatility so that the particles never escape from a compact domain. 
Although it largely intensifies the computations, especially in the analysis of the geometric Brownian motions, we still successfully obtained reasonable estimates that guarantee the exponential decay of~\eqref{e:2nd-deri-u}.

\subsection{Organization of this paper}

We present the full setting of the problem and our main theorem (Theorem~\ref{t:main}) in Section~\ref{s:main-result}.
After that we expand the details of the decompositions with master equations and the linearized Fokker-Planck equations.
It then follows the proof of the main theorem, modulo the core estimates on the solutions to the linearized Fokker-Planck equations.
In Section~\ref{s:core-est}, we present the generic properties of the linearized Fokker-Planck equations and then prove the core estimates for their solutions, with some details left to Appendix~\ref{s:complete-prf}.
The rest of this paper (Section~\ref{s:proof-core} and Appendix~\ref{s:technical}) verifies the technical results in Section~\ref{s:core-est}.

\subsection{Notation}

Throughout this paper, we will mainly work in the $d$-dimensional real Euclidean space $\R^d$ equipped with the Euclidean 2-norm $\abs{x}_2 \defeq \sum_{i=1}^d x_i^2$. 
A closed ball $B(c,r)$ in $\R^d$ is the set $\{x \in \R^d: \abs{x-c}_2 \le r\}$. 
The subscript 2 is usually ignored when the context is clear.

For a function $f$ defined on $\R^d$, we define its partial derivatives by 
\begin{equation*}
    \p^k f = \p_{x_1}^{k_1} \cdots \p_{x_d}^{k_d} f
\end{equation*}
for every multi-index $k = (k_1, \dots, k_d) \in \N^d$.
The size of the multi-index $k$ is measured in its Euclidean 1-norm $\abs{k}_1 \defeq \sum_{i=1}^d k_i$.
The subscript 1 is usually ignored when it is clear that $k$ is a multi-index.
Given a domain $\Omega \subseteq \R^d$, the Sobolev space $W^{n,\infty}(\Omega)$ contains the functions $f: \Omega \to \R$ with bounded derivatives up to order $n$. 
Then the Sobolev norm is then defined as
\begin{equation*}
	\norm{f}_{n,\infty} \defeq \sum_{\abs{k}_1 \le n} \norm{\p^k f}_{L^\infty(\Omega)} \,,
\end{equation*}
with the norm of its dual space $(W^{n,\infty}(\Omega))'$ being
\begin{equation*}
	\norm{q}_{(n,\infty)'} \defeq \sup_{\norm{f}_{n,\infty} \le 1} \ip{f, q} \,.
\end{equation*}
The domain $\Omega$ is usually taken to be $\R^d$ or some closed ball $B(c,r) \subset \R^d$.
Here we denote the linear operation of a distribution $q$ in the dual space $(W^{n,\infty}(\Omega))'$ on a function $f \in W^{n,\infty}(\Omega)$ by 
\begin{equation*}
    \ip{f, q} \defeq \int_{\Omega} f(x) q(dx) \,.
\end{equation*}
The Fourier transformation of any $q \in (W^{n,\infty}(\R^d))'$ is then defined by 
\begin{equation*}
    \f q(\xi) = \ip{e^{-i2\pi \xi^\top (\cdot)}, q} = \int_{\R^d} e^{-i2\pi \xi^\top x} q(dx) \,, \qquad \forall \xi \in \R^d \,.
\end{equation*}

The space of time-continuous functions over the time interval $[t_1, t_2]$ taking values in the metric space $\x$ is denoted by $\c([t_1, t_2]; \x)$, where $\x$ is usually the Euclidean space $\R^d$ or the Sobolev dual space $(W^{n,\infty}(\R^d))'$ in the context of this paper. 

The set of probability measures on a domain $\Omega \subseteq \R^d$ is denoted by $\prob(\Omega)$. 
The subset $\prob_2(\Omega) \subset \prob(\Omega)$ contains all elements of $\prob(\Omega)$ with finite second moments.
For a function $\vfn$ taking values from $\prob(\R^d)$, we define its derivatives with respect to the measure argument in the following way, as seen in~\cite{CardaliaguetDelarueLasryLions2019,CarmonaDelarue2018I}. 
\begin{definition}
	A function $\vfn: \prob(\R^d) \to \R$ has a \emph{continuous directional derivative} at some $\mu \in \prob(\R^d)$ if there exists a bounded continuous function $\frac{\delta \vfn}{\delta m}(\mu, \cdot): \R^d \to \R$ such that 
	for any $\tilde \mu \in \prob(\R^d)$, 
	\begin{equation*}
		\left. \frac{d}{d\ep} \right\vert_{\ep=0^+} \vfn(\mu + \ep(\tilde \mu - \mu)) = \int_{\R^d} \frac{\delta \vfn}{\delta m} (\mu, y) (\tilde \mu - \mu)(dy) \,.
	\end{equation*}
	We call $\frac{\delta \vfn}{\delta m}$ the \emph{linear functional derivative} of $\vfn$. 
	It is unique up to an additive constant, so we may enforce 
	\begin{equation*}
		\int_{\R^d} \frac{\delta \vfn}{\delta m}(\mu, y) \mu(dy) = 0 \,, \qquad \forall \mu \in \prob(\R^d) \,.
	\end{equation*}
	Analogously, for $p \ge 2$, we define $(\mu, y_p) \mapsto \frac{\delta^p \vfn}{\delta m^p}(\mu, y_1,\dots,y_p)$ to be the linear functional derivative of $\frac{\delta^{p-1} \vfn}{\delta m^{p-1}}(\cdot, y_1,\dots,y_{p-1})$ for all $y_1,\dots,y_{p-1} \in \R^d$. 
\end{definition}

\section{Main result and strategy}
\label{s:main-result}

\subsection{Setting of the CBO method}

Objective functions rising from real-world applications are typically non-convex but admit some degree of smoothness and polynomial growth away from minima. 
Throughout the paper, we focus on those functions $\efn: \R^d \to \R$ satisfying the following conditions. 
\begin{condition}
	\label{cd:efn-property}
	\begin{enumerate}
		\item There exists a unique global minimizer $x^\ast$ of $\efn$, and $\underline{\efn} \defeq \efn(x^\ast) \ge 0$, $\grad \efn(x^\ast) = 0$. 
		\item There exists some constant $c_\efn$ such that 
		\begin{equation*}
			\efn(x) - \efn(x^\ast) \ge c_\efn \abs{x - x^\ast}^2 \,, \qquad \forall x \in \R^d \,.
		\end{equation*} 

		\item The function $\efn$ has derivatives up to the 4th order, and there exists some constant $\ell_\efn$ such that $\norm{\p^k \efn}_\infty \le \ell_\efn$ for all multi-index $k$ such that $\abs{k}_1 = 2$. 
	\end{enumerate}
\end{condition}
Recall that the CBO method is a dynamic searching algorithm over the space $\R^d$, while it saves resources to restrict the search to certain compact domains. 
Thus it is reasonable to have \textit{a priori} information about the scope of searching. 
We then make an additional assumption on the objective function. 
\begin{assumption}
	\label{as:efn-restrict}
	It is known to the user that the (unique) global minimizer $x^\ast$ lies within some closed ball $B(c_0, r_\cut)$. 
\end{assumption}

With the \textit{a priori} information, one can safely run the CBO method without consider the region exterior to $B(c_0, 2r_\cut)$. 
This allows us to introduce a smooth cutoff function $\phi_{c,2r} \in C^\infty(\R^d)$ such that 
\begin{equation*}
	\phi_{c,2r}(x) = \begin{cases}
		1 \,, & \abs{x-c}_2 \le r \,, \\
		0 \,, & \abs{x-c}_2 > 2r \,,
	\end{cases}
\end{equation*} 
and for every multi-index $k \in \N^d$ with $\abs{k}_1 \le 6$ there exists some constant $c_k > 0$ such that 
\begin{equation*}
    \norm{\p^k \phi_{c,2r}}_\infty \le \frac{c_k}{r^{\abs{k}_1}} \,.
\end{equation*}
Note that $c_k = 1$ for $\abs{k}_1 = 0$.
We abbreviate $\phi_{c,2r} = \phi $ when the context is clear. 
Then the modified algorithm runs the cutoff system 
\begin{equation}
	\label{e:CBO-particles-cutoff}
	d X^{N,i}_t = - \lambda (X^{N,i}_t - m^N_t) dt + \sigma \abs{X^{N,i}_t - m^N_t} \phi_{c_0,2r_\cut}(X^{N,i}_t) dW^i_t 
\end{equation}
for $i=1,\dots,N$, with the consensus $m^N_t$ and the empirical measure $\nu^N_t$ defined as in~\eqref{e:CBO-particles}, and the initial data are given by $X^N_0 \sim \nu_\init^{\otimes N}$ with $\supp(\nu_\init) \subset B(c_0, r_\cut)$.
It is not hard to see that the above system is always restricted to the compact domain $B(c_0, 2r_\cut)$, as interpreted in the following proposition. 
\begin{proposition}
	\label{pp:CBO-restrict}
    For any $N \in \Z_+$, the system~\eqref{e:CBO-particles-cutoff} admits a unique strong solution on the time interval $[0,\infty)$, and 
    \begin{equation*}
        \max_{1 \le i\le N} \abs{X^{N,i}_t - c_0} \le 2r_\cut
    \end{equation*}
	almost surely for all $t \ge 0$. 
\end{proposition}
The purpose of the cutoff function, additional to the original CBO method in~\cite{CarriloChoiTotzeckTse2018,FornasierKlockRiedl2021}, is to maintain a bounded support for the empirical distributions $\nu^N_t$.
This is crucial to the estimates of the quantities to be introduced in the rest of this section.

\subsection{Weak propagation of chaos}

Note that $\prob(\Omega)$ can be viewed as a subspace of any Sobolev dual space $(W^{n,\infty}(\Omega))'$ for $\Omega = B(c_0, 2r_\cut)$ with $n \in \N$, where we recall the Sobolev norm 
\begin{equation*}
	\norm{f}_{n,\infty} \defeq \sum_{\abs{k} \le n} \norm{\p^k f}_{L^\infty(B(c_0, 2r_\cut))} \,,
\end{equation*}
and the dual norm 
\begin{equation*}
	\norm{\mu}_{(n,\infty)'} = \sup_{\norm{f}_{n,\infty} \le 1} \ip{f, \mu} \,.
\end{equation*}

The goal of this work is to show the uniform-in-time proximity of the empirical distributions $(\nu^N_t)_{t \ge 0}$ to its mean-field limit process $(\bar \nu_t)_{t \ge 0} \in \c([0,\infty); \prob(B(c_0, 2r_\cut)))$, a time-continuous weak solution to the stochastic differential equation  
\begin{equation}
	\label{e:CBO-mean-field}
	d X_t = -\lambda (X_t - M(\mu_t)) dt + \sigma \abs{X_t - M(\mu_t)} \phi(X_t) d W_t \,, \qquad \mu_t = \law(X_t) \,,
\end{equation}
where $M: \prob(\R^d) \to \R^d$ is the \emph{consensus operator}
\begin{equation*}
	M: \mu \mapsto \frac{\int x e^{-\alpha \efn(x)} \mu(dx)}{\int e^{-\alpha \efn(x)} \mu(dx)} \,.
\end{equation*}
Equivalently, one may view $(\bar \nu_t)_{t \ge 0}$ as a solution to the (nonlinear) \emph{Fokker-Planck equation}, interpreted in the weak sense that
\begin{equation}
	\label{e:CBO-fokker-planck}
	\frac{d}{dt} \int f(x) \mu_t(dx) = \int \left( {b(x,\mu_t)}^\top \grad f(x) + \frac{1}{2} \tr [a(x, \mu_t) \grad^2 f(x)] \right) \mu_t(dx) 
\end{equation}
holds for all $f \in W^{2,\infty}$, with 
\begin{equation*}
	b(x, \mu) \defeq -\lambda(x - M(\mu)) \,, \qquad a(x, \mu) \defeq \sigma^2 \phi(x)^2 \abs{x-M(\mu)}^2 I_{d \times d} \,.
\end{equation*}
We denote by $\c_\cbo \subset \c([0,\infty); \prob(B(c_0, 2r_\cut)))$ the set of time-continuous weak solutions to the above Fokker-Planck equations. 

The behaviors of~\eqref{e:CBO-mean-field} has been well studied in~\cite{CarriloChoiTotzeckTse2018,FornasierKlockRiedl2021}, and in this work we make use of only the most important decay property of the corresponding weak solution.

\begin{proposition}
	\label{pp:CBO-mf-decay}
	Suppose $\efn: \R^d \to \R$ satisfies Condition~\ref{cd:efn-property}, and 
    \begin{equation*}
        \lambda \ge d \sigma^2 e^{18c_\efn \alpha r_\cut^2 - \alpha \underline{\efn}} + 4r_\cut^2 e^{9c_\efn \alpha r_\cut^2} \,.
    \end{equation*}
    Then, for any $\mu_0 \in \prob(B(c_0,2r_\cut))$ there exists a unique weak solution $(\mu_t)_{t \ge 0} \in \c_\cbo$ to~\eqref{e:CBO-fokker-planck} and a point $\tilde x^\mu \in B(c_0, 2r_\cut)$ such that, for all $t \ge 0$,
	\begin{equation*}
		\abs{\int x \mu_t(dx) - \tilde x^\mu} \le C e^{-\kappa t}  \qquad \text{and} \qquad \abs{M(\mu_t) - \tilde x^\mu} \le C e^{-\kappa t}
	\end{equation*}
	with $\kappa = 2(\lambda - d\sigma^2 e^{9\alpha c_\efn r_\cut^2 - \alpha \underline{\efn}} )$ and $C$ independent on $\mu$.
\end{proposition}
\begin{remark}
    The above result is an adaptation of Theorem 4.1, \cite{CarriloChoiTotzeckTse2018}.
    Thanks to the boundedness of the domain $B(c_0, 2r_\cut)$, we can synthesize the conditions on the initial data to a single inequality. 
    It simply requires $\lambda$ to be large enough.
	
    Note that the limit point $\tilde x^\mu$ need not be the global optimizer $x^\ast$. 
    The deviation of $\tilde x^\mu$ from $x^\ast$ can be controlled by choosing some large $\alpha$, as shown in Theorem 4.2, \cite{CarriloChoiTotzeckTse2018}.
\end{remark}

Recall that the proximity is measured in the form $\Phi(\nu^N_t) - \Phi(\bar \nu_t)$. 
Now we state the following result on the weak propagation of chaos. 
\begin{theorem}[Main Theorem]
\label{t:main}
    Let $(\nu^N_t)_{t \ge 0}$ be the empirical distribution of the particle system~\eqref{e:CBO-particles-cutoff}, and let $(\bar \nu_t)_{t \ge 0}$ be the weak solution to the mean-field system~\eqref{e:CBO-mean-field} with $\bar \nu_0 = \nu_{\init}$.
    Suppose $\Phi: \prob(\R^d) \to \R$ satisfies the following conditions:
    \begin{enumerate}
        \item there exists some constant $C_\Phi$ such that 
        \begin{equation}
        \label{e:phi-reg-6}
            \sup_{\mu\in \prob(\R^d)} \norm{\frac{\delta \Phi}{\delta m}(\mu)}_{6,\infty} < C_\Phi \,, \qquad \sup_{\mu\in \prob(\R^d)} \norm{\frac{\delta^2 \Phi}{\delta m^2}(\mu)}_{6,\infty} < C_\Phi \,,
        \end{equation}
        \item and 
        \begin{equation}
        \label{e:phi-translate-inv}
            \left. \frac{d}{dz} \right\vert_{z=0} \Phi(\mu \circ \tau_z^{-1}) = 0 \,,
        \end{equation}
        where $\tau_z: x \mapsto x+z$ is the translation operator.
    \end{enumerate}
    Also assume that $\supp(\nu_\init) \subset B(c_0,r_\cut)$. 
    Then for all $N \ge 2$ we have 
    \begin{equation*}
        \sup_{t \ge 0} \abs{ \E[\Phi(\nu^N_t)] - \Phi(\bar \nu_t) } \le \frac{C_\main}{N} \,,
    \end{equation*}
    where $C_{\main}$ is a constant independent of $\nu_\init$.
\end{theorem}
\iffalse 
\begin{theorem}[main theorem]
\label{t:main}
    Consider the particle system~\eqref{e:CBO-particles-cutoff} with empirical distributions $\nu^N_t = \frac{1}{N} \sum_{j=1}^N \delta_{X^{N,j}_t}$, and let $(\bar \nu_t)_{t \ge 0}$ be the weak solution to the mean-field limit~\eqref{e:CBO-mean-field} with $\bar \nu_0 = \nu_{\init}$.
    Suppose $\Phi: \prob(\R^d) \to \R$ satisfies one of the following conditions:
    \begin{enumerate}
        \item $\Phi(\mu) = \norm{\mu - \delta_{x^\ast}}_{-s,2}^2$ with $s > \frac{d}{2} + 3$ \textcolor{red}{(need not hold anymore)};

        \item there exists some constant $C_\Phi$ such that 
        \begin{equation}
        \label{e:phi-reg-6}
            \sup_{\mu\in \prob(\R^d)} \norm{\frac{\delta \Phi}{\delta m}(\mu)}_{6,\infty} < C_\Phi \,, \qquad \sup_{\mu\in \prob(\R^d)} \norm{\frac{\delta^2 \Phi}{\delta m^2}(\mu)}_{6,\infty} < C_\Phi \,,
        \end{equation}
        and 
        \begin{equation}
        \label{e:phi-translate-inv}
            \left. \frac{d}{dz} \right\vert_{z=0} \Phi(\mu \circ \tau_z^{-1}) = 0 \,.
        \end{equation}
    \end{enumerate}
    Also assume that $\supp(\nu_\init) \subset B(c_0,r_\cut)$. 
    Then there exists some constant $C_{\main}$, depending only on $(\lambda, \sigma, \alpha, \efn, c_0, r_\cut)$, such that 
    \begin{equation*}
        \sup_{t \ge 0} \abs{ \E[\Phi(\nu^N_t)] - \Phi(\bar \nu_t) } \le \frac{C_\main}{N} \,.
    \end{equation*}
\end{theorem}
\fi 

As a consequence of Theorem~\ref{t:main}, we obtain some estimates on the weak errors between $\nu^N_t$ and $\delta_{x^\ast}$, in the Wasserstein-type metrics, with convergence rates in $N$ and $t$ separately. 
Here we present two examples as corollaries of the main theorem.
The proofs will be given in Section~\ref{s:proof-main}.

\begin{corollary}[Convergence in centered Fourier-Wasserstein distance]
\label{co:conv-in-t-N}
    Let $s > \frac{d+7}{2}$ and $\id: \R^d \to \R^d$ the identity map such that $\id(x) = x$ for all $x \in \R^d$.
    There exists some constant $C_{FW} > 0$ such that 
    \begin{equation*}
        \E \norm{\nu^N_t \circ \tau_{\ip{\id, \nu^N_t}} - \delta_0}_{-s,2}^2 \le C_{FW} (N^{-1} + e^{-2\kappa t} )
    \end{equation*}
    for all $N \ge 2$ and $t \ge 0$.
\end{corollary}

\begin{corollary}[Convergence in centered Wasserstein distance]
\label{co:conv-w2}
    Define the Wasserstein 2-distance $\was_2: \prob_2(\R^d) \times \prob_2(\R^d) \to \bar [0, \infty)$ by 
    \begin{equation*}
        \was_2(\mu, \nu)^2 \defeq \inf_{\pi \in \Pi(\mu,\nu)} \iint \abs{x-y} \pi(dx,dy) \,,
    \end{equation*}
    where $\Pi(\mu,\nu)$ is the set of all couplings of $\mu$ and $\nu$. 
    The identity map $\id: \R^d \to \R^d$ is defined as in Corollary~\ref{co:conv-in-t-N}.
    Then there exists some constant $C_{W} > 0$ such that 
    \begin{equation*}
        \E \was_2(\nu^N_t \circ \tau_{\ip{\id, \nu^N_t}}, \delta_0)^2 \le C_W(N^{-1} + e^{-\kappa t}) 
    \end{equation*}
    for all $N \ge 2$ and $t \ge 0$.
\end{corollary}

\begin{remark}
    We make a few comments on the above corollaries.
    \begin{enumerate}
        \item The upper bounds tend to 0 as $N, t \to \infty$ in both corollaries. 
        This means the particles cluster to the global optimizer when there are a large number of them and the algorithm runs for a sufficiently long time.

        \item The Wasserstein distance is in general non-smooth on the space $\prob_2(\R^d)$. 
        But since our target is the convergence towards $\delta_0$, the unique centered attracting invariant measure, there is a clear explicit formula for the centered Wasserstein distance.
        Its regularity still admits the conditions of our main theorem, which gives Corollary~\ref{co:conv-w2}, a stronger argument than Corollary~\ref{co:conv-in-t-N}.
        
        \item Notice that the errors due to the finiteness of the particle numbers $N$ and the running time $t$ appear separately. 
        The choice of $N$ can be independent of $t$ when given an error tolerance in practice, which improves the allocation of computational resources.
    \end{enumerate}
\end{remark}

\subsection{Strategies of the proof}
\label{s:strategy}

Observe that the Fokker-Planck equation actually defines a solution operator $P_t$ that pushes $\nu_\init$ to $\bar \nu_t$ for all $t \ge 0$. 
In general, this sequence of operators map any $\mu \in \prob(\R^d)$ to the corresponding probability flow $(\mu_t)_{t \ge 0}$ that solves~\eqref{e:CBO-fokker-planck}.
Thus, for any function $\Phi: \prob(\R^d) \to \R$, we may define $\ufn: [0,\infty) \times \prob(\R^d) \to \R$ via 
\begin{equation*}
	\ufn(t, \mu) = \Phi(P_t \mu) \,,
\end{equation*}
and we often abbreviate $P_t \mu = \mu_t$ and sometimes abuse the notation $\mu$ for $(\mu_t)_{t \ge 0}$. 
Then the proximity between $\nu^N_t$ and $\bar \nu_t$ translates to
\begin{equation*}
	\E[\Phi(\nu^N_t)] - \Phi(\bar \nu_t) = \E [\ufn(0, \nu^N_t)] - \ufn(t, \nu_\init)  \,.
\end{equation*}
We shall analyze the behavior of the function $\ufn$ to obtain an estimate on the above quantity.

\subsubsection{Decomposition via the master equation}

We may decompose the objective by 
\begin{equation}
\label{e:u-decompose}
	\ufn(0, \nu^N_t) - \ufn(t, \nu_\init) = (\cU(t,\nu^N_0) - \cU(t,\nu_\init)) + (\cU(0,\nu^N_t) - \cU(t,\nu^N_0)) \,.
\end{equation}
Notice that the first component in~\eqref{e:u-decompose} is static in time. 
We apply Theorem 2.14 of~\cite{ChassagneuxSzpruchTse2022} to see that 
\begin{align}
	\nonumber
    & \E [\cU(t,\nu^N_0) - \cU(t,\nu_\init)] = \\ 
	\label{e:static-bounded}
	& \qquad \frac{1}{N} \int_0^1 \int_0^1 \E \left[ s_1 \frac{\delta^2 \cU}{\delta m^2}(t,\tilde \mu^N_{s_1,s_2}, \tilde \eta, \tilde \eta) - s_1 \frac{\delta^2 \cU}{\delta m^2}(t, \tilde \mu^N_{s_1,s_2}, \tilde \eta, X^{N,1}_0) \right] ds_1 ds_2 \,,
\end{align}
with $\tilde \eta \sim \nu_\init$ independent of $(X^{N,j}_0)_{j=1}^N$ and 
\begin{equation*}
    \tilde \mu^N_{s_1,s_2} \defeq (1-s_1) \nu_{\init} + s_1 \left( \nu^N_0 + \frac{s_2}{N} (\delta_{\tilde\eta} - \delta_{X^{N,1}_0}) \right) \,.
\end{equation*}

For the second component in~\eqref{e:u-decompose}, we notice that the function $\ufn: [0,\infty) \times \prob(\R^d)$ as previously defined solves the \emph{master equation}~\cite{BuckdahnLiPengRainer2017,CarmonaDelarue2018II} 
\begin{equation}
	\label{e:master}
	\p_t \ufn(t,\mu) = \int_{\R^d} \paren[\Big]{ {b(x,\mu)}^\top \grad_x \frac{\delta \ufn}{\delta m}(t,\mu,x) + \frac{1}{2} \tr[ a(x,\mu) \grad^2_x \frac{\delta \ufn}{\delta m}(t,\mu,x)  ]  } \mu(dx) \,.
\end{equation}
Then, by It\^{o}'s formula and the master equation, we get 
\begin{align}
    \nonumber
    & \E [\cU(0,\nu^N_t) - \cU(t,\nu^N_0)] = \\
    \label{e:need-exp-decay}
    & \qquad \frac{1}{2N} \sum_{i=1}^d \int_0^t \E \left[ \int_{\R^d}  a(z, \nu^N_s) \p_{(x_2)_i} \p_{(x_1)_i} \frac{\delta^2 \cU}{\delta m^2} (t-s,\nu^N_s,z,z) \nu^N_s(dz)  \right] ds \,.
\end{align}

Our target is to find a uniform-in-time upper bound on $\frac{\delta^2 \ufn}{\delta m^2}(t,\cdot,\cdot,\cdot)$ and an exponential decay in time of $\p_{(x_2)_i} \p_{(x_1)_i} \frac{\delta^2 \ufn}{\delta m^2}(t,\cdot,\cdot,\cdot)$.
Those will put uniform-in-time upper bounds on~\eqref{e:static-bounded} and~\eqref{e:need-exp-decay}, which finally leads to the conclusion. 
To analyze the two objects, we express them as solutions to some linear differential equations and study their evolution in time.

\subsubsection{Linearized Fokker-Planck equation}

Recall the Fokker-Planck equation 
\begin{equation*}
	\pa_t \mu_t = \lin_{\mu_t}^\ast \mu_t \,,
\end{equation*}
where 
\begin{equation*}
    \lin_{\mu} f(x) \defeq \frac{\sigma^2}{2} \abs{x - M(\mu)}^2 \phi(x)^2 \lap f(x) - \lambda (x-M(\mu))^\top \grad f(x)  \,.
\end{equation*}
Note that $\lin_{\mu_t}^\ast \mu_t$ depends nonlinearly on $\mu_t$. 

As in~\cite{Tse2021}, we consider a linearization of the Fokker-Planck equation in the form
% by differentiating the expression $\ip{f, \lin_{\mu}^\ast \mu}$ in $\mu$ to get the linear expressions
\begin{align*}
	q \mapsto & \frac{1}{2} \int \tr (a(x,\mu) \grad^2 f(x) ) q(dx) + \int b(x,\mu)^\top \grad f(x) q(dx) \\
	& + \frac{1}{2} \int \tr \left( \frac{\delta a}{\delta m}(x,\mu) (q) \grad^2 f(x) \right) \mu(dx) + \int {\frac{\delta b}{\delta m} (x,\mu)(q)}\cdot \grad f(x)  \mu(dx) \\
	& + \text{(remainder)} \,,
\end{align*}
where 
\begin{equation*}
	\frac{\delta b}{\delta m} (x,\mu)(q) = \int \frac{\delta b}{\delta m} (x,\mu,y) q(dy) \,, \qquad \frac{\delta a}{\delta m}(x,\mu) (q) = \int \frac{\delta a}{\delta m}(x,\mu,y) q(dy) \,, 
\end{equation*}
with 
\begin{align*}
	\frac{\delta b}{\delta m} (x,\mu,y) & = \frac{\lambda (y-M(\mu)) e^{-\alpha \efn(y)}}{\ip{e^{-\alpha\efn},\mu}} \,, \\
	\frac{\delta a}{\delta m}(x,\mu,y) & = - \frac{2 \sigma^2 \phi(x)^2 e^{-\alpha \cE(y)} (x-M(\mu))^\top (y-M(\mu))}{\ip{e^{-\alpha \cE}, \mu}} I_{d \times d} \,.
\end{align*}
So we define the linearization operator
\begin{equation}
\label{e:operator-a}
    \a_{\mu} f(x) \defeq \ip{\lambda \grad f - \sigma^2 \phi^2(\cdot)(\cdot - M(\mu)) \lap f, \mu} \cdot (x-M(\mu)) \frac{e^{-\alpha\efn(x)}}{\ip{e^{-\alpha\efn}, \mu}} \,.
\end{equation}
Then, given the solution $(\mu_t)_{t \ge 0}$ to~\eqref{e:CBO-fokker-planck}, an initial state $q_0 \in (W^{n,\infty})'$, and the remainder terms $(r_t)_{t \ge 0} \in \bigcap_{T>0} L^\infty([0,T]; (W^{n,\infty})')$, 
we define the \emph{linearized Fokker-Planck equation} with $\mu = (\mu_t)_{t \ge 0} \in \c_\cbo$ by 
\begin{equation}
\label{e:l-fpe}
\tag{L-FPE}
    \p_t q_t = \lin_{\mu_t}^\ast q_t + \a_{\mu_t}^\ast q_t + r_t \,, 
\end{equation}
where the solution $(q_t)_{t \ge 0}$ is denoted by $\cauchy{\mu}{q_0}{r} \in \bigcap_{T>0} L^\infty([0,T]; (W^{n,\infty})')$.

Then we make use of Propositions 3.3 and 3.4 of~\cite{DelarueTse2021} to express the derivatives of $\ufn$ in terms of the solutions to some particular L-FPEs. 
\begin{lemma}
	\label{l:decomp-into-m}
    Suppose $\Phi$ satisfies the condition of Theorem~\ref{t:main}. 
    Then, for any $\mu \in \c_\cbo$, $z_1, z_2 \in \R^d$, we have 
    \begin{equation*}
	\frac{\delta^2 \cU}{\delta m^2} (t,\mu_0,z_1,z_2) = \frac{\delta^2 \Phi}{\delta m^2} (\mu_t) (m^{(1)}(t;\mu,\delta_{z_1}), m^{(1)}(t;\mu,\delta_{z_2})) + \frac{\delta \Phi}{\delta m} (\mu_t) (m^{(2)}(t;\mu,\delta_{z_1}, \delta_{z_2})) \,,
    \end{equation*}
	where 
	\begin{align*}
		m^{(1)}(\cdot;\mu,\nu) & := \cauchy{\mu}{\nu-\mu_0}{0} \,, \\
		m^{(2)}(\cdot;\mu,\nu_1,\nu_2) & := \cauchy{\mu}{\mu_0-\nu_2}{\src{\mu}{m^{(1)}(\cdot;\mu,\nu_1)}{m^{(1)}(\cdot;\mu,\nu_2)} } \,,
	\end{align*}
	with 
	\begin{align*}
		\int f(x) \src{\mu}{q_1}{q_2}(t) (dx) & = \int \grad f(x) \cdot \frac{\delta b}{\delta m}(x,\mu_t)(q_2(t)) q_1(t)(dx) \\
		& + \int \grad f(x) \cdot \frac{\delta b}{\delta m}(x,\mu_t)(q_1(t)) q_2(t)(dx) \\
		& + \int \grad f(x) \cdot \frac{\delta^2 b}{\delta m^2}(x,\mu_t)(q_1(t),q_2(t)) \mu_t(dx) \\
		& + \frac{1}{2} \int \tr[ \grad^2 f(x) \frac{\delta a}{\delta m}(x,\mu_t)(q_2(t)) ] q_1(t)(dx) \\
		& + \frac{1}{2} \int \tr[ \grad^2 f(x) \frac{\delta a}{\delta m}(x,\mu_t)(q_1(t)) ] q_2(t)(dx) \\
		& + \frac{1}{2} \int \tr[ \grad^2 f(x) \frac{\delta^2 a}{\delta m^2}(x,\mu_t)(q_1(t),q_2(t)) ] \mu_t(dx) \,.
	\end{align*}
\end{lemma}

\begin{lemma}
	\label{l:decomp-into-d}
	Suppose $\Phi$ satisfies the condition of Theorem~\ref{t:main}. 
    Then, for any $\mu \in \c_\cbo$, $z_1, z_2 \in \R^d$, $i,j \in [d]$, we have 
    \begin{align*}
        & \pa_{(z_2)_j} \pa_{(z_1)_i} \frac{\delta^2 \cU}{\delta m^2} (t,\mu_0,z_1,z_2) = \\
		& \qquad \frac{\delta^2 \Phi}{\delta m^2} (\mu_t) (d^{(1)}_i(t;\mu,{z_1}), d^{(1)}_j(t;\mu,{z_2})) + \frac{\delta \Phi}{\delta m} (\mu_t) (d^{(2)}_{i,j}(t;\mu_0,{z_1}, {z_2}))  \,,
    \end{align*}
	where 
	\begin{align*}
		d^{(1)}_j (\cdot; \mu,z) & := \cauchy{(\mu_t)_{t \ge 0}}{-\pa_{x_j} \delta_z}{0} \,, \\
		d^{(2)}_{i,j} (\cdot; \mu,z_1,z_2) & := \cauchy{(\mu_t)_{t \ge 0}}{0}{\src{(\mu_t)_{t \ge 0}}{d^{(1)}_i(\cdot;\mu,z_1)}{d^{(1)}_j(\cdot;\mu,z_2)}} \,.
	\end{align*}
	Here $\p_{x_j} \delta_z$ is the distribution such that $\ip{f, -\p_{x_j} \delta_z} \defeq \p_{x_j}f(z)$. 
\end{lemma}

The decomposition above reduces our goal to the bounds and decays of the those solutions to the L-FPEs. 
We present the estimates on those solutions in the following lemmata and defer the proofs to Sections~\ref{s:prf-1st-order} and~\ref{s:prf-2nd-order}.

The following two lemmata show the boundedness of $m^{(1)}$ and $d^{(1)}$ and the exponential decay of their projections. 
\begin{lemma}
\label{l:m1-bound}
	There exist some constants $C, \kappa_0 > 0$ such that, for every $j \in [d]$, $\mu \in \c_\cbo$, $z \in B(c_0, 2r_\cut)$, 
	\begin{equation*}
		\sup_{t \ge 0} \norm{m^{(1)}(t;\mu,\delta_z)}_{(2,\infty)'} \le C \,,
	\end{equation*}
	and one can find a vector $q^{(1)}_{0,\infty}(\mu,z) \in \R^d$ such that 
	\begin{equation*}
		\norm{m^{(1)}(t;\mu,\delta_z) - q^{(1)}_{0,\infty}(\mu,z) \cdot \grad \delta_{\tilde x^\mu} }_{(4,\infty)'} \le C e^{-\kappa_0 t} \,, \qquad \forall t \ge 0 \,.
	\end{equation*}
\end{lemma}
\begin{lemma}
\label{l:d1-bound}
	There exist some constants $C, \kappa_0 > 0$ such that, for every $j \in [d]$, $\mu \in \c_\cbo$, $z \in B(c_0, 2r_\cut)$, 
	\begin{equation*}
		\sup_{t \ge 0} \norm{d^{(1)}_j(t;\mu,z)}_{(2,\infty)'} \le C \,,
	\end{equation*}
	and one can find a vector $q^{(1)}_{j,\infty}(\mu,z) \in \R^d$ such that 
	\begin{equation*}
		\norm{d^{(1)}_j(t;\mu,z) - q^{(1)}_{j,\infty}(\mu,z) \cdot \grad \delta_{\tilde x^\mu} }_{(4,\infty)'} \le C e^{-\kappa_0 t} \,, \qquad \forall t \ge 0 \,.
	\end{equation*}
\end{lemma}

Similarly, the two lemmata below present the uniform boundedness and decay of $m^{(2)}$ and $d^{(2)}$.
\begin{lemma}
\label{l:m2-bound}
	For every $j \in [d]$, $\mu \in \c_\cbo$, $z_1,z_2 \in B(c_0, 2r_\cut)$, we define 
	\begin{equation*}
		\bar m^{(2)}(\cdot;\mu,\delta_{z_1},\delta_{z_2}) = m^{(2)}(\cdot;\mu,\delta_{z_1},\delta_{z_2}) - q^{(1)}_{0,\infty}(\mu,z_2) \cdot \grad m^{(1)}(\cdot;\mu,\delta_{z_1}) \,.
	\end{equation*}
	Then there exist some constants $C > 0$, uniformly in $j,\mu,z_1,z_2$, such that 
	\begin{equation*}
		\norm{\bar m^{(2)}(\cdot;\mu,\delta_{z_1},\delta_{z_2})}_{(6,\infty)'} \le C \,.
	\end{equation*}
\end{lemma}
\begin{lemma}
\label{l:d2-bound}
	For every $j \in [d]$, $\mu \in \c_\cbo$, $z_1,z_2 \in B(c_0, 2r_\cut)$, we define 
	\begin{equation*}
		\bar d^{(2)}_{j,j}(\cdot;\mu,z_1,z_2) = d^{(2)}_{j,j}(\cdot;\mu,z_1,z_2) - q^{(1)}_{j,\infty}(\mu,z_2) \cdot \grad d^{(1)}_j(\cdot;\mu,z_1) \,.
	\end{equation*}
	Then there exist some constants $C, \kappa_0' > 0$, uniformly in $j,\mu,z_1,z_2$, such that 
	\begin{equation*}
		\norm{\bar d^{(2)}_{j,j}(t;\mu,z_1,z_2)}_{(4,\infty)'} \le Ct \,,  \qquad \forall t \ge 0 \,,
	\end{equation*}
	and one can find a vector $q^{(2)}_{j,\infty}(\mu,z_1,z_2) \in \R^d$ such that 
	\begin{equation*}
		\norm{\bar d^{(2)}_{j,j}(t;\mu,z_1,z_2) - q^{(2)}_{j,\infty}(\mu,z_1,z_2) \cdot \grad \delta_{\tilde x^\mu} }_{(6,\infty)'} \le C e^{-\kappa_0' t} \,, \qquad \forall t \ge 0 \,.
	\end{equation*}
\end{lemma}

\begin{remark}
	\label{rm:q-infty-bound}
    The above lemmata also tell that there exists some constant $C$ such that 
    \begin{equation*}
        \sup_{j,\mu,z} \abs{q^{(1)}_{j,\infty}(\mu,z)} \le C \,, \qquad \sup_{j,\mu,z_1,z_2} \abs{q^{(2)}_{j,\infty}(\mu,z_1,z_2)} \le C \,.
    \end{equation*}
\end{remark}

\subsection{Proof of the main theorem}
\label{s:proof-main}

We now apply the previous key lemmata to prove our main results.

\begin{proof}[Proof of Theorem~\ref{t:main}]

    We need to show that 
    \begin{equation*}
        \abs{ \frac{\delta^2 \cU}{\delta m^2} (t,\mu,z_1,z_2) } \le C'
    \end{equation*}
    and 
    \begin{equation*}
        \abs{ \p_{(z_1)_j} \p_{(z_2)_j} \frac{\delta^2 \cU}{\delta m^2} (t,\mu,z_1,z_2) } \le C' e^{-\kappa_0' t}
    \end{equation*}
    uniformly for all $t \ge 0$, $\mu \in \prob(B(c_0, 2r_\cut))$, and $z_1, z_2 \in B(c_0, 2r_\cut)$.
    
    Recall that $\norm{\mu_t - \delta_{\tilde x^\mu}}_{(1,\infty)'}$ admits exponential decay in $t$, due to Proposition~\ref{pp:CBO-mf-decay}, so Lemmata~\ref{l:m1-bound}-\ref{l:d2-bound} hold also when $\delta_{\tilde x^\mu}$ is replaced by $\mu_t$. 
    Explicitly, 
    \begin{equation*}
        \norm{m^{(1)}(t;\mu,\delta_z) - q^{(1)}_{0,\infty}(\mu,z) \cdot \grad \mu_t }_{(4,\infty)'} \le C e^{-\kappa_0 t} \,, 
    \end{equation*}
    \begin{equation*}
        \norm{d^{(1)}_j(t;\mu,z) - q^{(1)}_{j,\infty}(\mu,z) \cdot \grad \mu_t }_{(4,\infty)'} \le C e^{-\kappa_0 t} \,, 
    \end{equation*}
    and 
    \begin{equation*}
        \norm{\bar d^{(2)}_{j,j}(t;\mu,z_1,z_2) - q^{(2)}_{j,\infty}(\mu,z_1,z_2) \cdot \grad \mu_t }_{(6,\infty)'} \le C e^{-\kappa_0' t} \,.
    \end{equation*}
    We denote for short by 
    \begin{equation*}
        \tilde m^{(1)}(t;\mu,\delta_z) \defeq m^{(1)}(t;\mu,\delta_z) - q^{(1)}_{0,\infty}(\mu,z) \cdot \grad \mu_t  \,,
    \end{equation*}
    \begin{equation*}
        \tilde d^{(1)}_j(t;\mu,z) \defeq d^{(1)}_j(t;\mu,z) - q^{(1)}_{j,\infty}(\mu,z) \cdot \grad \mu_t  \,,
    \end{equation*}
    and 
    \begin{equation*}
        \tilde d^{(2)}_{i,j}(t;\mu,z_1,z_2) \defeq \bar d^{(2)}_{j,j}(t;\mu,z_1,z_2) - q^{(2)}_{j,\infty}(\mu,z_1,z_2) \cdot \grad \mu_t 
    \end{equation*}
    accordingly.

    We adopt Lemma 4.11 of~\cite{DelarueTse2021} to see that,
    \begin{equation}
        \frac{\delta \Phi}{\delta m}(\mu) (\pa_{x_j} \mu) = 0 \,, \qquad \frac{\delta^2 \Phi}{\delta m^2}(\mu) (\pa_{x_j} \mu, q) + \frac{\delta \Phi}{\delta m}(\mu)(\pa_{x_j} q) = 0 
    \end{equation}
    for all $j \in [d]$, $\mu \in \prob(\R^d)$, and $q \in (W^{1,\infty})'$ such that $\ip{q,1} = 0$.
    Then, similar to Proposition 4.12, \cite{DelarueTse2021}, we have 
    \begin{align*}
        & \frac{\delta^2 \cU}{\delta m^2} (t,\mu,z_1,z_2) \\
        & \qquad = \frac{\delta^2 \Phi}{\delta m^2}(\mu_t) (\tilde m^{(1)}(t;\mu,\delta_{z_1}), m^{(1)}(t;\mu,\delta_{z_2})) + \frac{\delta \Phi}{\delta m} (\mu_t) (\bar m^{(2)} (t;\mu,\delta_{z_1},\delta_{z_2})) \,.
    \end{align*}
    Applying Lemmata~\ref{l:m1-bound} and~\ref{l:m2-bound} to the two components respectively gives 
    \begin{align*}
        & \abs{\frac{\delta^2 \Phi}{\delta m^2}(\mu_t) (\tilde m^{(1)}(t;\mu,\delta_{z_1}), m^{(1)}(t;\mu,\delta_{z_2}))} \\
        & \le \norm{\frac{\delta^2 \Phi}{\delta m^2}(\mu_t) }_{6,\infty} \norm{\tilde m^{(1)}(t;\mu,\delta_{z_1})}_{(4,\infty)'} \norm{m^{(1)}(t;\mu,\delta_{z_2})}_{(2,\infty)'} \\
        & \le C_\Phi C^2 
    \end{align*}
    and 
    \begin{align*}
        & \abs{\frac{\delta \Phi}{\delta m} (\mu_t) (\bar m^{(2)}(t;\mu,\delta_{z_1},\delta_{z_2}))} \\
        & \le \norm{\frac{\delta \Phi}{\delta m} (\mu_t)}_{6,\infty} \norm{ \bar m^{(2)}(t;\mu,\delta_{z_1},\delta_{z_2}) }_{(6,\infty)'} \le C_\Phi C \,.
    \end{align*}
    So
    \begin{equation*}
        \abs{ \frac{\delta^2 \cU}{\delta m^2} (t,\mu,z_1,z_2) } \le C_\Phi C (1+C) 
    \end{equation*}
    for all $t \ge 0$. 
    Plugging into~\eqref{e:static-bounded} gives 
    \begin{equation*}
        \abs{\E[\ufn(t,\nu^N_0) - \ufn(t,\nu_{\init})]} \le \frac{C'}{N} \,.
    \end{equation*}
    
    Analogously, we apply Proposition 4.12, \cite{DelarueTse2021}, and~\eqref{e:phi-translate-inv} again to see that 
    \begin{align*}
        & \p_{(z_1)_j} \p_{(z_2)_j} \frac{\delta^2 \cU}{\delta m^2} (t,\mu,z_1,z_2) \\
        & \qquad = \frac{\delta^2 \Phi}{\delta m^2}(\mu_t) (\tilde d^{(1)}_j (t;\mu,{z_1}), d^{(1)}_j (t;\mu,{z_2})) + \frac{\delta \Phi}{\delta m} (\mu_t) (\bar d^{(2)}_{j,j} (t;\mu,{z_1},{z_2})) \\
        & \qquad = \frac{\delta^2 \Phi}{\delta m^2}(\mu_t) (\tilde d^{(1)}_j (t;\mu,{z_1}), d^{(1)}_j (t;\mu,{z_2})) + \frac{\delta \Phi}{\delta m} (\mu_t) (\tilde d^{(2)}_{j,j} (t;\mu,{z_1},{z_2})) \,.
    \end{align*}
    With the same adjustment as above applied to Lemmata~\ref{l:d1-bound} and~\ref{l:d2-bound}, we have
    \begin{align*}
        & \abs{\frac{\delta^2 \Phi}{\delta m^2}(\mu_t) (\tilde d^{(1)}_j(t;\mu,{z_1}), d^{(1)}_j(t;\mu,{z_2}))} \\
        & \le \norm{\frac{\delta^2 \Phi}{\delta m^2}(\mu_t) }_{6,\infty} \norm{\tilde d^{(1)}_j(t;\mu,{z_1})}_{(4,\infty)'} \norm{d^{(1)}_j(t;\mu,{z_2})}_{(2,\infty)'} \\
        & \le C_\Phi C^2 e^{-2\kappa_0 t}
    \end{align*}
    and 
    \begin{align*}
        & \abs{\frac{\delta \Phi}{\delta m} (\mu_t) (\tilde d^{(2)}_{j,j}(t;\mu,{z_1},{z_2}))} \\
        & \le \norm{\frac{\delta \Phi}{\delta m} (\mu_t)}_{6,\infty} \norm{ \tilde d^{(2)}_{j,j}(t;\mu,{z_1},{z_2}) }_{(6,\infty)'} \le C_\Phi C e^{-\kappa_0' t} \,.
    \end{align*}
    So 
    \begin{equation*}
        \abs{ \p_{(z_1)_j} \p_{(z_2)_j} \frac{\delta^2 \cU}{\delta m^2} (t,\mu,z_1,z_2) } \le C_\Phi C (1+C) e^{-\kappa_0' t} 
    \end{equation*}
    for all $t \ge 0$. 
    Plugging into~\eqref{e:need-exp-decay} gives 
    \begin{equation*}
        \abs{ \E[\ufn(0,\nu^N_t) - \ufn(t,\nu^N_0)] } \le \frac{d\sigma^2 16 r_\cut^2 C_\Phi C (1+C)}{2 \kappa_0' N} \le \frac{C'}{N} \,.
    \end{equation*}
    
    Joining the two parts, we hence obtain
    \begin{equation*}
        \sup_{t \ge 0} \abs{ \E[\Phi(\nu^N_t)] - \Phi(\bar\nu_t) } \le \frac{C_\main}{N} \,. 
    \end{equation*}
\end{proof}

Using the main theorem, we may justify the explicit examples.
\begin{proof}[Proof of Corollary~\ref{co:conv-in-t-N}]
    We first show that a function $\Phi: \prob(\R^d) \to \R$ defined by 
    \begin{equation*}
        \Phi(\mu) \defeq \norm{\mu \circ \tau_{\ip{\id,\mu}} - \rho_\refer}_{-s,2}^2 
    \end{equation*}
    satisfies the conditions~\eqref{e:phi-reg-6} and~\eqref{e:phi-translate-inv} in Theorem~\ref{t:main}.

    Observe that, for any $\mu \in \prob(\R^d)$ and $z \in \R^d$,
    \begin{equation*}
        \ip{\id, \mu \circ \tau_z^{-1}} - z = \ip{\id, \mu} \,,
    \end{equation*}
    so $\Phi(\mu \circ \tau_z^{-1})$ is invariant of $z$, which gives~\eqref{e:phi-translate-inv}.

    To show the smoothness, we notice that 
    \begin{align*}
        \Phi(\mu) & = \int (1+\abs{\xi}^2)^{-s} \left[ \paren[\Big]{ \int \cos(2\pi \xi \cdot (x - \ip{\id, \mu})) \mu(dx) - \int \cos(2\pi \xi \cdot x) \rho_\refer(dx) }^2 \right. \\
        & \qquad \left. + \paren[\Big]{ \int \sin(2\pi \xi \cdot (x - \ip{\id, \mu})) \mu(dx) - \int \sin(2\pi \xi \cdot x) \rho_\refer(dx) }^2  \right] d\xi \,,
    \end{align*}
    which leads to 
    \begin{align*}
        \frac{\delta \Phi}{\delta m}(\mu,x) & = \int 2(1+\abs{\xi}^2)^{-s} \left( \int \cos(2\pi \xi \cdot (y-x)) \mu(dy) \right. \\
        & \qquad \qquad - \int \cos(2 \pi \xi \cdot (y-x+\ip{\id, \mu})) \rho_\refer(dy) \\
        & \qquad \left. + 2\pi \xi \cdot x \iint \sin(2\pi\xi \cdot (y_2-y_1+\ip{\id, \mu})) \mu(dy_1) \rho_\refer(dy_2) \right) d\xi \,.
    \end{align*}
    Note that cosine and sine are bounded smooth functions, and $x$ is restricted to $B(c_0, 2r_\cut)$. 
    This implies
    \begin{equation*}
        \sup_\mu \norm{\frac{\delta \Phi}{\delta m}(\mu)}_{6,\infty} \lesssim \int (1+\abs{\xi}^2)^{-s} \abs{\xi}^6 d\xi < \infty \,.
    \end{equation*}
    Similarly, we obtain 
    \begin{equation*}
        \sup_\mu \norm{\frac{\delta^2 \Phi}{\delta m^2}(\mu)}_{6,\infty} \lesssim \int (1+\abs{\xi}^2)^{-s} \abs{\xi}^7 d\xi < \infty \,.
    \end{equation*}
    It thus verifies~\eqref{e:phi-reg-6}.
    
    Now, taking $\rho_\refer = \delta_0$, we see that
    \begin{equation*}
        \E \norm{\nu^N_t \circ \tau_{\ip{\id, \nu^N_t}} - \delta_0}_{-s,2}^2 \le \frac{2C_\main}{N} + 2\norm{\bar\nu_t \circ \tau_{\ip{\id, \bar\nu_t}} - \delta_0}_{-s,2}^2 \,.
    \end{equation*}
    By the definition of Sobolev norm and then triangular inequality, we have
    \begin{align*}
        \norm{\bar\nu_t \circ \tau_{\ip{\id, \bar\nu_t}} - \delta_0}_{-s,2} & = \norm{\bar\nu_t  - \delta_{\ip{\id,\bar\nu_t}} }_{-s,2} \\
        & \le \norm{\bar\nu_t  - \delta_{\tilde x^{\nu_\init}} }_{-s,2} + \norm{\delta_{\ip{\id,\bar\nu_t}} - \delta_{\tilde x^{\nu_\init}} }_{-s,2} \,,
    \end{align*}
    where $\tilde x^{\nu_\init}$ is the limit point given in Proposition~\ref{pp:CBO-mf-decay}.
    Applying the convergence result in Proposition~\ref{pp:CBO-mf-decay}, we know there exists some constant $C > 0$ such that 
    \begin{equation*}
        \norm{\bar\nu_t \circ \tau_{\ip{\id, \bar\nu_t}} - \delta_0}_{-s,2} \le C e^{-\kappa t} \,.
    \end{equation*}
    Therefore, taking $C_{FW} = \max \{2C_\main, 2C^2\} $ leads to 
    \begin{equation*}
        \E \norm{\nu^N_t \circ \tau_{\ip{\id, \nu^N_t}} - \delta_0}_{-s,2}^2 \le C_{FW}  (N^{-1} + e^{-2\kappa t}) \,.
    \end{equation*}
\end{proof}

\begin{proof}[Proof of Corollary~\ref{co:conv-w2}]
    Notice that $\prob(B(c_0, 2r_\cut))$ can be viewed as a subset of $\prob_2(\R^d)$.
    Let $\Phi: \prob(B(c_0, 2r_\cut)) \to \R$ be such that $\Phi(\mu) = \was_2(\mu \circ \tau_{\ip{\id, \mu}}, \delta_0)$. 
    Observe that 
    \begin{equation*}
        \was_2(\mu \circ \tau_{\ip{\id, \mu}}, \delta_0)^2 = \int \abs{x - \ip{\id, \mu}}^2 \mu(dx) \,,
    \end{equation*}
    which is the variance of $\mu$, so it easily verifies~\eqref{e:phi-translate-inv}.
    Also, one may compute that $\frac{\delta \Phi}{\delta m}(\mu,x) = \abs{x - \ip{\id, \mu}}^2$ and $\frac{\delta^2 \Phi}{\delta m^2}(\mu,x,y) = -2 (x-\ip{\id, \mu}) \cdot y$.
    Those are bounded for $\mu \in \prob(B(c_0, 2r_\cut))$ and $x,y \in B(c_0, 2r_\cut)$, which verifies~\eqref{e:phi-reg-6}. 
    Then, applying Theorem~\ref{t:main} and triangular inequality, we have 
    \begin{equation*}
        \E \was_2(\nu^N_t \circ \tau_{\ip{\id, \nu^N_t}}, \delta_0)^2 \le \frac{2C_\main}{N} + 2\was_2(\bar \nu_t \circ \tau_{\ip{\id, \bar \nu_t}}, \delta_0)^2 \,.
    \end{equation*}
    Recall that $\was_2(\bar \nu_t \circ \tau_{\ip{\id, \bar \nu_t}}, \delta_0)^2$ is the variance of $\bar \nu_t$, and thus by Theorem 4.1, \cite{CarriloChoiTotzeckTse2018}, we know 
    \begin{equation*}
        \was_2(\bar \nu_t \circ \tau_{\ip{\id, \bar \nu_t}}, \delta_0)^2 \le e^{-\kappa t} \was_2(\nu_\init \circ \tau_{\ip{\id, \nu_\init}}, \delta_0)^2 \,.
    \end{equation*}
    Hence we conclude that 
    \begin{equation*}
        \E \was_2(\nu^N_t \circ \tau_{\ip{\id, \nu^N_t}}, \delta_0)^2 \le C_W (N^{-1} + e^{-\kappa t}) \,.
    \end{equation*}
\end{proof}

\section{Core estimates for the proofs}
\label{s:core-est}

%Recall that we have reduced to the case $x^\ast = 0$ with $\efn(0) = 0$. 
In this section, we present the essential steps towards Lemmata~\ref{l:m1-bound}-\ref{l:d2-bound}. 
The core of the estimates, those with exponential decay in particular, has the following form,
\begin{equation}
	\label{e:erg-abs}
	\norm{q_t - q_\infty \cdot \grad \delta_{\tilde x}}_{(n,\infty)'} \le C e^{-\kappa_0 t} \,,
	\tag{Erg}
\end{equation}
where $(q_t)_{t \ge 0}$ is the solution to some $\cauchy{\mu}{q_0}{r}$. 
We say $(q_t)_{t \ge 0}$ \emph{satisfies~\eqref{e:erg-abs} at $\mu$} in the above case.

Fix $\mu \in \c_\cbo$ with attracting invariant measure $\mu_\infty = \delta_{\tilde{x}^\mu}$.
The strategy to expand the Sobolev dual norm 
\begin{equation*}
	\norm{q_T}_{(n,\infty)'} = \sup_{\norm{\xi}_{n,\infty} \le 1} \ip{\xi, q_T} = \sup_{\norm{\xi}_{n,\infty}} \int_{B(c_0,2r_\cut)} \xi(x) q_T(dx)
\end{equation*}
is to consider a backward Cauchy problem 
\begin{equation}
	\label{e:backward-cauchy}
	\begin{cases}
		\p_t w_t + \lin_{\mu_t} w_t + \a_{\mu_t} w_t = 0 \,, & t \in [0,T) \,, \\
		w_T = \xi \,,
	\end{cases}
\end{equation}
with solution $w = w^{T,\xi} \in \c([0,T]; W^{n,\infty}(B(c_0,2r_\cut)))$. 
Then, for any $T > 0$ and $\xi \in W^{n,\infty}$, we have 
\begin{equation*}
	\ip{\xi, q_T} = \ip{w^{T,\xi}(0), q_0} + \int_0^T \ip{w^{T,\xi}(t), r_t} dt \,.
\end{equation*}
So the objectives become 
\begin{equation*}
	\norm{q_T}_{(n,\infty)'} = \sup_{\norm{\xi}_{n,\infty} \le 1} \left( \ip{w^{T,\xi}(0), q_0} + \int_0^T \ip{w^{T,\xi}(t), r_t} dt \right) 
\end{equation*}
and similarly
\begin{align}
    \nonumber
	& \norm{q_T - q_\infty \cdot \grad \delta_{\tilde x^\mu}}_{(n+2,\infty)'} = \\ 
    \label{e:normalize-w}
    & \qquad \sup_{\norm{\xi}_{n,\infty} \le 1} \left( \ip{w^{T,\xi}(0), q_0} + \int_0^T \ip{w^{T,\xi}(t), r_t} dt + q_\infty \cdot \grad \xi(\tilde x^\mu) \right) \,.
\end{align}

Thus, the estimates reduce to the properties of the backward equation~\eqref{e:backward-cauchy}.
We study its behaviors in the rest of this section and present the most general result regarding its connection to~\eqref{e:erg-abs}.

\subsection{Analysis of generic L-FPE}

In this subsection, we present some generic results that will lead to the core lemmata in Section~\ref{s:strategy}. 
The proofs of these results are postponed to Section~\ref{s:proof-core}.

Note that the L-FPE $\p_t q_t = L_{\mu_t} q_t + r_t$, where $L_{\mu_t} \defeq \lin_{\mu_t}^\ast + \a_{\mu_t}^\ast$, can be rewritten as 
\begin{equation}
	\label{e:generic-LFPE-reduce}
	\p_t q_t = L_{\delta_{\tilde x^\mu}} q_t + \paren[\big]{ (L_{\mu_t} - L_{\delta_{\tilde x^\mu}}) q_t + r_t } \,.
\end{equation}
Then $(q_t)_{t \ge 0}$ also solves $\cauchy{\delta_{\tilde x^\mu}}{q_0}{((L_{\mu_t} - L_{\delta_{\tilde x^\mu}}) q_t + r_t)_{t \ge 0}}$. 
Thus we shall first consider the case where $\mu = \delta_{\tilde x^\mu}$ by showing that the $q_\infty$ term in~\eqref{e:normalize-w} cancels with the non-exponential components in front of it, so that the overall sum displays an exponential decay. 
Then we reduce the general case to this particular one.

\begin{lemma}[(Erg) at invariant measure]
	\label{lm:erg-at-delta}
	Let $n \in \{2,4\}$. 
	Suppose $\mu_0 = \delta_{\tilde x^\mu}$, and $q_0 \in (W^{n,\infty})'$ is a linear combination of the following terms:
	\begin{equation*}
		-\p_{x_j} \delta_z \,, \qquad \p_{x_j}^2 \delta_z \,, \qquad (\delta_z -\nu) \,,
	\end{equation*}
	with $j \in [d]$, $z \in B(c_0, 2r_\cut)$, $\nu \in \prob(B(c_0, 2r_\cut))$, where $r_\cut$ is the (half) cut-off radius of $\phi$.
	Suppose also that $(r_t)_{t \ge 0}$ has the form 
	\begin{equation*}
		r_t = \grad \cdot R_t \defeq \sum_{i=1}^d \p_{x_i} R^i_t
	\end{equation*}
	for some distributions $R^1_t, \dots, R^d_t \in (W^{n_0,\infty})'$, $n_0 = n+1$, such that for every $i \in [d]$,
	\begin{equation*}
		\norm{R^i_t}_{(n_0,\infty)'} \le K e^{-\lambda_0 t} \,, \qquad \forall t \ge 0
	\end{equation*} 
	for some constants $K, \lambda_0 > 0$. 
	Then there exist constants $C, \kappa_1 > 0$, independent of $(q_t)_{t \ge 0} = \cauchy{\mu}{q_0}{r}$, such that one can find a vector $q_\infty \in \R^d$ with which 
	\begin{equation}
		\label{e:exp-decay-general}
		\norm{q_t - q_\infty \cdot \grad\delta_{\tilde x^\mu} }_{(n+2, \infty)'} \le C (1+K) e^{-\kappa_1 t}
	\end{equation}
	holds for all $t \ge 0$. 
\end{lemma}

\begin{remark}
	The above lemma actually shows the~\eqref{e:erg-abs} property of $m^{(1)}(\cdot;\delta_{\tilde x^\mu},z)$ and $d^{(1)}_i(\cdot;\delta_{\tilde x^\mu},z)$. 
\end{remark}

The next step is to verify that the remainder terms in~\eqref{e:generic-LFPE-reduce}, including both $r^1_t \defeq (L_{\mu_t} - L_{\delta_{\tilde x^\mu}}) q_t$ and $r_t$, fulfill the pre-conditions of Lemma~\ref{lm:erg-at-delta}.
Note that $r$ is defined for particular $q$, so we first look at $r^1$, where $r^1_t = (L_{\mu_t} - L_{\delta_{\tilde x^\mu}}) q_t$. 
The following two lemmata present the relation between $r^1_t$ and $q_t$, and then the decay of $r^1_t$ due to the uniform bound on $q_t$. 

\begin{lemma}
	\label{lm:connect-r1-q}
	Let $n \in \{2,4\}$. 
	Suppose $q = \cauchy{\mu}{q_0}{r} \in \bigcap_{T>0} L^\infty([0,T]; (W^{n,\infty})')$ with $\mu \in \c_\cbo$. 
	Define $r^1_t = (L_{\mu_t} - L_{\delta_{\tilde x^\mu}}) q_t$ for $t \ge 0$. 
	Then one can express $r^1_t$ as 
	\begin{equation*}
		r^1_t = - \sum_{i=1}^d \p_{x_i} R^{1,i}_t 
	\end{equation*}
	for some $R^{1,i} \in \bigcap_{T>0} L^\infty([0,T]; (W^{n+1,\infty})')$, $i = 1,\dots,d$, and there exists some constant $\bar C > 0$, depending only on $\lambda, \sigma, \alpha, \efn$, and $\phi$, such that 
	\begin{equation*}
		\sup_{i \in [d]} \norm{R^{1,i}_t}_{(n+1,\infty)'} \le \bar C e^{-\kappa t} \norm{q_t}_{(n,\infty)'} \,, 
	\end{equation*}
	where $\kappa$ is given as in Proposition~\ref{pp:CBO-mf-decay}.
\end{lemma}

It remains to obtain a uniform upper bound on $\norm{q_t}_{(n,\infty)'}$, which will also prove the first inequalities in Lemmata~\ref{l:m1-bound} and~\ref{l:d1-bound}.
The details are shown in the following lemma. 
\begin{lemma}
	\label{lm:bound-q}
	Let $n \in \{2,4\}$. 
	Suppose $q = \cauchy{\mu}{q_0}{r} \in \bigcap_{T>0} L^\infty([0,T]; (W^{n,\infty})')$ with $\mu \in \c_\cbo$, $q_0 \in (W^{n,\infty})'$ is a linear combination of the following terms:
	\begin{equation*}
		-\p_{x_j} \delta_z \,, \qquad \p_{x_j}^2 \delta_z \,, \qquad (\delta_z -\nu) \,.
	\end{equation*}
	Assume that $(r_t)_{t \ge 0}$ has the form 
	\begin{equation*}
		r_t = \grad \cdot R_t \defeq \sum_{i=1}^d \p_{x_i} R^i_t
	\end{equation*}
	for some distributions $R^1_t, \dots, R^d_t \in (W^{n_0,\infty})'$ such that for every $i \in [d]$,
	\begin{equation*}
		\norm{R^i_t}_{(n_0,\infty)'} \le K e^{-\lambda_0 t} \,, \qquad \forall t \ge 0
	\end{equation*} 
	for some constants $K, \lambda_0 > 0$. 
	Then, there exists some constant $C$, independent of $q_0$ and $\mu$, such that 
	\begin{equation*}
		\sup_{t \ge 0} \norm{q_t}_{(n \lor (1+n_0),\infty)} \le C (1 + K) \,.
	\end{equation*}
\end{lemma}

\subsection{Proof of Lemmata~\ref{l:m1-bound} and~\ref{l:d1-bound}}
\label{s:prf-1st-order}

With the above results, we are able to proceed to the core lemmata.
In this subsection, we present the proofs for the estimates on the first-order terms $m^{(1)}$ and $d^{(1)}$.

\begin{proof}[Proof of Lemma~\ref{l:m1-bound}]
	Recall that $m^{(1)}(\cdot;\mu,\delta_z) = \cauchy{\mu}{\delta_z - \mu}{0}$.
	Applying Lemma~\ref{lm:bound-q} with $R^i_t = 0$ (so $n_0=0$) for all $i \in [d]$ gives 
	\begin{equation*}
		\sup_{t \ge 0} \norm{m^{(1)}(t;\mu,\delta_z)}_{(2,\infty)'} \le C \,,
	\end{equation*}
	where $C$ is independent of $\mu$ and $z$, which proves the first inequality of the Lemma. 

	Joining Lemma~\ref{lm:bound-q} and Lemma~\ref{lm:connect-r1-q}, we see that there is actually an indifferent uniform decay 
	\begin{equation*}
		\norm{R^{1,i}_t}_{(1+n,\infty)'}  \le C \bar C e^{-\kappa t} \,, \qquad t \ge 0 \,.
	\end{equation*}
	Thus the pre-condition of Lemma~\ref{lm:erg-at-delta} is satisfied with $K = C \bar C$. 
	Hence we know there exists some $q^{(1)}_{0,\infty}(\mu,z) \in \R^d$ such that for all $t \ge 0$
	\begin{equation*}
		\norm{ m^{(1)}(t;\mu,\delta_z) - q^{(1)}_{0,\infty}(\mu,z) \cdot \grad \delta_{\tilde x^\mu} }_{(4,\infty)'} \le C e^{-\kappa_0 t} \,.
	\end{equation*}
\end{proof}

\begin{proof}[Proof of Lemma~\ref{l:d1-bound}]
	Recall that $d^{(1)}_j(\cdot;\mu,z) = \cauchy{\mu}{-\p_{x_j} \delta_z}{0}$.
	Apply Lemmata~\ref{lm:erg-at-delta}, \ref{lm:connect-r1-q} and~\ref{lm:bound-q} in the same as in the Proof of Lemma~\ref{l:m1-bound}, we get 
	\begin{equation*}
		\sup_{t \ge 0} \norm{d^{(1)}_j(t;\mu,z)}_{(2,\infty)'} \le C \,,
	\end{equation*}
	and, for all $t \ge 0$,
	\begin{equation*}
		\norm{ d^{(1)}_j(t;\mu,z) - q^{(1)}_{j,\infty}(\mu,z) \cdot \grad \delta_{\tilde x^\mu} }_{(4,\infty)'} \le C e^{-\kappa_0 t} \,.
	\end{equation*}
\end{proof}

\subsection{Proof of Lemmata~\ref{l:m2-bound} and~\ref{l:d2-bound}}
\label{s:prf-2nd-order}

Now we prove the estimates on the second-order terms $m^{(2)}$ and $d^{(2)}$, which are the essential components in the decomposition of the main object of study.
The crux of the proofs is to verify the pre-condition on the remainder terms for Lemmata~\ref{lm:erg-at-delta} and~\ref{lm:bound-q}.
There are numerous computational details in the proof, so we provide a sketch here and attach the full proofs in Appendix~\ref{s:complete-prf}
\begin{proof}[Sketch of proof of Lemma~\ref{l:m2-bound}]
	Fix $\mu \in \c_\cbo$ and $z_1, z_2 \in B(0, 3r_\cut)$.
	Recall that 
	\begin{equation*}
		\bar m^{(2)}(t;\mu,\delta_{z_1},\delta_{z_2}) = m^{(2)}(t;\mu,\delta_{z_1},\delta_{z_2}) - q^{(1)}_{0,\infty}(\mu,z_2) \cdot \grad m^{(1)}(t;\mu,\delta_{z_1}) \,,
	\end{equation*}
	where $m^{(2)}(\cdot;\mu,\delta_{z_1},\delta_{z_2}) = \cauchy{\mu}{\mu_0 - \delta_{z_2}}{(r_t)_{t \ge 0}}$ with 
	\begin{equation*}
		(r_t)_{t \ge 0} = \src{\mu}{m^{(1)}(\cdot;\mu,\delta_{z_1})}{m^{(1)}(\cdot;\mu,\delta_{z_2})} \,.
	\end{equation*}

	We may view $\bar m^{(2)}$ as also a solution to some L-FPE 
	\begin{equation*}
		\p_t \bar m^{(2)}(t;\mu,\delta_{z_1},\delta_{z_2}) = (\lin_{\mu_t}^\ast + \a_{\mu_t}^\ast) \bar m^{(2)}(t;\mu,\delta_{z_1},\delta_{z_2}) + \bar r(t;\mu,\delta_{z_1},\delta_{z_2}) \,,  
	\end{equation*}
	with $\bar r$ defined by 
	\begin{equation*}
		\bar r(t) \defeq r(t) + q^{(1)}_{0,\infty}(\mu,z_2) \cdot \paren[\Big]{ (\lin_{\mu_t}^\ast + \a_{\mu_t}^\ast) \grad m^{(1)}(t;\mu,\delta_{z_1}) - \grad (\lin_{\mu_t}^\ast + \a_{\mu_t}^\ast) m^{(1)}(t;\mu,\delta_{z_1}) } \,.
	\end{equation*}
	We show that $\bar r$ can be expressed as 
	\begin{equation*}
		\ip{f, \bar r(t)} = \sum_{i=1}^d \ip{g_i, \bar R^i(t)} 
	\end{equation*}
	for some $R^i(t)$, $i \in [d]$, such that $\norm{R^i(t)}_{(5,\infty)'} \le K e^{-\kappa_0 t}$ for all $i \in [d]$ and $t \ge 0$. 
	Then the pre-condition of Lemma~\ref{lm:bound-q} on $r$ is verified. 
	Notice that 
	\begin{equation*}
		\bar m^{(2)}(0;\mu,\delta_{z_1},\delta_{z_2}) = (\mu_0 - \delta_{z_2}) - q^{(1)}_{0,\infty}(\mu,z_2) \cdot \grad (\delta_{z_1} - \mu_0)) \,.
	\end{equation*}
	A slight generalization of Lemma~\ref{lm:bound-q} gives the overall bound 
	\begin{equation*}
		\norm{\bar m^{(2)}(t;\mu,\delta_{z_1},\delta_{z_2})}_{(6,\infty)'} \le C(1+K) \,.
	\end{equation*}
\end{proof}

An analogous strategy leads to Lemma~\ref{l:d2-bound}.

\begin{proof}[Sketch of proof of Lemma~\ref{l:d2-bound}]
	Fix $\mu \in \c_\cbo$, $j \in [d]$, and $z_1, z_2 \in B(c_0, 2r_\cut)$.
	Recall that 
	\begin{equation*}
		\bar d^{(2)}_{j,j}(t;\mu,{z_1},{z_2}) = d^{(2)}_{j,j}(t;\mu,{z_1},_{z_2}) - q^{(1)}_{j,\infty}(\mu,z_2) \cdot \grad d^{(1)}_j(t;\mu,{z_1}) \,,
	\end{equation*}
	where $d^{(2)}_{j,j}(\cdot;\mu,{z_1},{z_2}) = \cauchy{\mu}{0}{(r_t)_{t \ge 0}}$ with 
	\begin{equation*}
		(r_t)_{t \ge 0} = \src{\mu}{d^{(1)}_j(\cdot;\mu,{z_1})}{d^{(1)}_j(\cdot;\mu,{z_2})} \,.
	\end{equation*}
	We aims at proving that
	\begin{equation}
		\label{e:d2-linear-growth}
		\norm{\bar d^{(2)}_{j,j}(t;\mu,z_1,z_2)}_{(4,\infty)'} \le C t  \,, \qquad \forall t \ge 0 \,,
	\end{equation}
	and there exists some $q^{(2)}_{j,\infty}(\mu,z_1,z_2) \in \R^d$ such that 
	\begin{equation}
		\label{e:d2-expdecay}
		\norm{\bar d^{(2)}_{j,j}(t;\mu,z_1,z_2) - q^{(2)}_{j,\infty}(\mu,z_1,z_2) \cdot \grad \delta_0 }_{(6,\infty)'} \le C e^{-\kappa_0' t}
	\end{equation}
	for all $t \ge 0$.

	Analogous to the proof of Lemma~\ref{l:m2-bound}, we may treat $\bar d^{(2)}_{j,j}(\cdot;\mu,{z_1},{z_2})$ as 
	\begin{equation*}
		\cauchy{\mu}{q^{(1)}_{j,\infty}(\mu,z_2) \cdot \grad \p_{x_j} \delta_{z_1}}{\bar r}
	\end{equation*}
	with 
	\begin{equation*}
		\bar r(t) \defeq r(t) + q^{(1)}_{j,\infty}(\mu,z_2) \cdot \paren[\Big]{ (\lin_{\mu_t}^\ast + \a_{\mu_t}^\ast) \grad d^{(1)}_j(t;\mu,{z_1}) - \grad (\lin_{\mu_t}^\ast + \a_{\mu_t}^\ast) d^{(1)}_j(t;\mu,{z_1}) } \,.
	\end{equation*}
	By replacing all $m^{(1)}(\cdot;\mu,\delta_{z_\iota})$ with $d^{(1)}_j(\cdot;\mu,z_\iota)$, $\iota = 1,2$, in the proof of Lemma~\ref{l:m2-bound}, 
	we see that $\bar r$ can be expressed as 
	\begin{equation*}
		\ip{f, \bar r(t)} = \sum_{i=1}^d \ip{g_i, \bar R^i(t)} 
	\end{equation*}
	for some $R^i(t)$, $i \in [d]$, such that 
	\begin{equation}
		\label{e:d2-r-decay}
		\sup_{i \in [d]} \norm{R^i(t)}_{(5,\infty)'} \le K e^{-\kappa_0 t}
	\end{equation}
	and 
	\begin{equation}
		\label{e:d2-r-bounded}
		\sup_{i \in [d]} \norm{R^i(t)}_{(3,\infty)'} \le K 
	\end{equation}
	for all $t \ge 0$.

	Putting~\eqref{e:d2-r-bounded} to the framework of Lemma~\ref{lm:bound-q}, we obtain~\eqref{e:d2-linear-growth}.
	Then, using~\ref{e:d2-r-decay}, we apply Lemmata~\eqref{lm:erg-at-delta} and~\ref{lm:connect-r1-q}, together with~\eqref{e:d2-linear-growth}, which gives us~\eqref{e:d2-expdecay}.
\end{proof}

\section{Proofs for Section~\ref{s:core-est}}
\label{s:proof-core}

\begin{proof}[Proof of Lemma~\ref{lm:erg-at-delta}]
	\step[Expression of $w$]
	Fix $T > 0$, $\xi \in W^{n,\infty}$, and $\tilde x = \tilde x^\mu$.
	Observe that the backward equation~\eqref{e:backward-cauchy} becomes 
	\begin{equation*}
		\begin{cases}
			\p_t w(t,x) + \frac{\sigma^2}{2} \phi(x)^2 \abs{x - \tilde x}^2 \lap w(t,x) - \lambda (x - \tilde x) \cdot \grad w(t,x) \\
            \qquad \qquad + \lambda e^{-\alpha (\efn(x) - \efn(\tilde x))} (x - \tilde x) \cdot \grad w(t,\tilde x) = 0 \,, \\
			w(T,x) = \xi(x) \,.
		\end{cases}
	\end{equation*}
	Taking derivatives of the first equation with respect to every $x_j$ and then plugging in $x = \tilde x$, we obtain explicitly
	\begin{equation*}
		\grad w(t,\tilde x) =  \grad \xi(\tilde x) \,.
	\end{equation*}
	Then it further transforms to a differential equation 
	\begin{equation*}
		\begin{cases}
			\p_t w(t,x) + \lin_{\delta_{\tilde x}} w(t,x) + \lambda e^{-\alpha (\efn(x) - \efn(\tilde x))} (x - \tilde x) \cdot \grad \xi(\tilde x)  = 0 \,, \\
			w(T,x) = \xi(x) \,.
		\end{cases}
	\end{equation*}
	Applying the Feynman-Kac formula, we have 
	\begin{equation*}
		w(t,x) = \E^{t,x} \left[ \xi(S_T) + \int_t^T \lambda e^{-\alpha (\efn(S_u) - \efn(\tilde x))} (S_u - \tilde x) \cdot \grad \xi(0) du \right] \,,
	\end{equation*}
	where $(S_u)_{u \in [t,T]}$ has dynamics $dS_u = -\lambda S_u du + \sigma \abs{S_u} \phi(S_u) dW_u$ with $S_t = x$ under the probability measure $\P^{t,x}$.
    Note that $S_u$ remains in $B(c_0, 2r_\cut)$ for $u \in [t,T]$ as long as $x \in B(c_0, 2r_\cut)$, so $w(t)$ is well-defined as an element of $W^{n,\infty}(B(c_0, 2r_\cut))$.

	We may write the above equation in a more generic form, 
	\begin{equation*}
		w(t,x) = \E [ \xi(S^{t,x}_{T}) ] + \lambda \grad\xi(\tilde x) \cdot \int_t^T  \E \left[G_{\tilde x} (S^{t,x}_u) \right] du \,,
	\end{equation*}
	where $dS^{t,x}_u = -\lambda S^{t,x}_u du + \sigma \abs{S^{t,x}_u} \phi(S^{t,x}_u) dW_u$ for $u \in [t,T]$, with $S^{t,x}_t = x$, and $G_{\tilde x}: y \mapsto (y-\tilde x) e^{-\alpha (\efn(y) - \efn(\tilde x)) }$. 
	We will use the properties of the process $S^{t,x}$ in each case separately. 

	\step[Case $q_0 = -\p_{x_j} \delta_z$]
	Notice that 
	\begin{align*}
		\p_{x_i} w(t,x) & = \E \left[ (\p_{x_i} S^{t,x}_T) \cdot \grad \xi (S^{t,x}_T) \right] \\
		& \qquad + \lambda \grad\xi(0) \cdot \int_t^T  \E \left[ D G_{\tilde x} (S^{t,x}_u) (\p_{x_i} S^{t,x}_u) \right] du \,,
	\end{align*}
    where $D G_{\tilde x}(y)$ is the Jacobian matrix of $G_{\tilde x}$ at $y$. 
    
	Define 
	\begin{equation*}
		\varphi_i^\rem(T_0;t_0,x) \defeq \lambda \int_{t_0}^{T_0} \E \left[  D G_{\tilde x} (S^{t,x}_u) (\p_{x_i} S^{t,x}_u) \right] du \,,
	\end{equation*}
	so that 
	\begin{align*}
		\ip{\xi, q_T} & = \ip{w(0), q_0} + \int_0^T \ip{w(t), r_t} dt \\
		& = \p_{x_j} w(0,z) - \int_0^T \sum_{i=1}^d \ip{\p_{x_i} w(t), R^i_t} dt \\
		& = \E \left[ (\p_{x_j} S^{0,z}_{T}) \cdot \grad \xi (S^{0,z}_{T}) \right] - \int_0^T \sum_{i=1}^d \ip{ \E \left[ (\p_{x_i} S^{t,\cdot}_T) \cdot \grad \xi (S^{t,\cdot}_T) \right], R^i_t } \\
		& \quad + \left( \varphi_j^\rem(T;0,z) - \int_0^T \sum_{i=1}^d \ip{ \varphi_i^\rem(T;t,\cdot), R^i_t } dt \right) \cdot \grad \xi(\tilde x) \,. 
	\end{align*}
	With Lemma~\ref{lm:prop-gbm}, we see that  
	\begin{align*}
		& \abs{\E \left[ (\p_{x_j} S^{0,z}_{T}) \cdot \grad \xi (S^{0,z}_{T}) \right] - \int_0^T \sum_{i=1}^d \ip{ \E \left[ (\p_{x_i} S^{t,\cdot}_T) \cdot \grad \xi (S^{t,\cdot}_T) \right], R^i_t }} \\
		& \le e^{-\lambda T/2} \norm{\grad \xi}_\infty + \int_0^T \sum_{i=1}^d \norm{\E \left[ (\p_{x_i} S^{t,\cdot}_T) \cdot \grad \xi (S^{t,\cdot}_T) \right]}_{n_0,\infty} \norm{R^i_t}_{(n_0,\infty)'} dt \\
		& \le e^{-\lambda T/2} \norm{\xi}_{1,\infty} + \int_0^T d C_0 e^{-\lambda (T-t)/2} K e^{-\lambda_0 t} \norm{\p_{x_i} \xi}_{n_0,\infty} dt \\
		& \le C e^{-\lambda_0 T} \norm{\xi}_{1+n_0,\infty} \,,
	\end{align*}
	where we may assume without loss of generality that $2\lambda_0 \le \lambda$.

	For the remaining terms, we first notice that the integrand in the definition of $\varphi^\rem$ is independent of $T_0$ and $\xi$. 
	So for $T_1 < T_2$ we have 
	\begin{equation*}
		\varphi_i^\rem(T_2; t,x) - \varphi_i^\rem(T_1; t,x) = \lambda \int_{T_1}^{T_2} \E \left[ D G_{\tilde x} (S^{t,x}_u) (\p_{x_i} S^{t,x}_u) \right] du \,,
	\end{equation*}
	which goes to $0$ in the rate $e^{-\lambda  (T_1-t)/2}$ as $T_1, T_2 \to \infty$, due to Lemma~\ref{lm:prop-gbm} and the boundedness of function $x \mapsto x e^{-\alpha \efn(x)}$.
	This means $\{\varphi_i^\rem(T_0;t,x)\}_{T_0 \ge t}$ is a Cauchy sequence and is thus convergent. 
	We denote the limit by $\varphi_{i,\infty}^\rem(t,x)$. 

	Now we define 
	\begin{equation*}
		\gamma_j^\rem(T_0; z) \defeq \varphi_j^\rem(T_0;0,z) - \int_0^{T_0} \sum_{i=1}^d \ip{ \varphi_i^\rem(T_0;t,\cdot), R^i_t } dt
	\end{equation*}
	Similarly, for $T_1 < T_2$, we have 
	\begin{align*}
		\gamma_j^\rem(T_2;z) - \gamma_j^\rem(T_1;z) & = \varphi_j^\rem(T_2;0,z) - \varphi_j^\rem(T_1;0,z) \\
		& \quad + \int_0^{T_1} \sum_{i=1}^d \ip{ \varphi_i^\rem(T_2;t,\cdot) - \varphi_i^\rem(T_1;t,\cdot), R^i_t } dt \\
		& \quad + \int_{T_1}^{T_2} \sum_{i=1}^d \ip{\varphi_i^\rem(T_2;t,\cdot), R^i_t} dt
	\end{align*}
	Using the convergence of $\varphi_i^\rem$ and also Lemma~\ref{lm:prop-gbm}, we get 
	\begin{equation*}
		\abs{ \gamma_j^\rem(T_2;z) - \gamma_j^\rem(T_1;z) } \le CK e^{-\lambda_0 T_1} \to 0 
	\end{equation*}
	as $T_1, T_2 \to \infty$. 
	Thus there exists some $\gamma_{j,\infty}^\rem(z)$ such that 
	\begin{equation*}
		\abs{ \gamma_j^\rem(T_0;z)  - \gamma_{j,\infty}^\rem(z) } \le CK e^{-\lambda_0 T_0} \,.
	\end{equation*}
	Plugging in back to the decomposition of $\ip{\xi, q_T}$, we conclude that 
	\begin{equation*}
		\ip{\xi, q_T - \gamma_{j,\infty}^\rem(z) \cdot \grad \delta_{\tilde x}} \le C e^{-\lambda_0 T} \norm{\xi}_{1+n_0,\infty} \,.
	\end{equation*}
	Taking $q_\infty = \gamma_{j,\infty}^\rem(z)$ gives~\eqref{e:exp-decay-general} for $q_0 = -\p_{x_j} \delta_z$ with $\kappa_1 = \lambda_0$. 

	\step[Case $q_0 = \p_{x_j}^2 \delta_z$]
	The analysis is analogous, albeit the first term $\ip{w(0), q_0}$ now becomes $\p_{x_j}^2 w(0,z)$. 
	Notice that 
	\begin{align*}
		& \p_{x_j}^2 w(t,x) = \E \left[ (\p_{x_j}^2 S^{t,x}_T) \cdot \grad \xi(S^{t,x}_T) + (\p_{x_j} S^{t,x}_T)^\top D^2 \xi(S^{t,x}_T) (\p_{x_j} S^{t,x}_T) \right] + \lambda \grad\xi(\tilde x) \cdot \\
		& \qquad \int_t^T  \E \left[ DG_{\tilde x} (S^{t,x}_u) (\p_{x_j}^2 S^{t,x}_u)   + (\p_{x_j} S^{t,x}_u)^\top D^2 G_{\tilde x} (S^{t,x}_u) (\p_{x_j} S^{t,x}_u) \right] du 
	\end{align*}
	Lemma~\ref{lm:prop-gbm} guarantees the corresponding exponential decay of $\p_{x_j}^2 S^{t,x}_T$ and $\abs{\p_{x_j} S^{t,x}_T}^2$. 
	In addition, the remainder term becomes 
	\begin{align*}
		& \gamma_j^\rem(T_0;z) = \\
		& \qquad \int_0^T  \E \left[ DG_{\tilde x} (S^{0,x}_t) (\p_{x_j}^2 S^{0,x}_t) + (\p_{x_j} S^{0,x}_t)^\top D^2 G_{\tilde x} (S^{0,x}_t) (\p_{x_j} S^{0,x}_t) \right] dt \\
		& \qquad - \int_0^{T} \sum_{i=1}^d \ip{ \varphi_i^\rem(T_0;t,\cdot), R^i_t } dt \,.
	\end{align*}
	Its boundedness and convergence follows from the analysis above. 
	Thus we also have~\eqref{e:exp-decay-general}.

	\step[Case $q_0 = \delta_z - \nu$]
	Similarly, it suffices to look at the term $\ip{w(0), q_0}$ as well. 
	With $q_0 = \delta_z - \nu$, we have 
	\begin{align*}
		\ip{w(0), q_0} = w(0,z) - \int w(0,x) \nu(dx) = \int (w(0,z) - w(0,x)) \nu(dx) \,.
	\end{align*}
	By mean-value theorem, for every $x$ there exists some $\zeta(z,x) \in B(c_0, 2 r_\cut)$ such that 
	\begin{align*}
		w(0,z) - w(0,x) & = \grad w(0, \zeta_{z,x}) \cdot (z-x) = \sum_{j=1}^d \pa_{x_j} w(0,\zeta_{z,x}) (z_j - x_j) \\
		& = \sum_{j=1}^d (z_j - x_j) \left( \E \left[ (\p_{x_j} S^{0,\zeta_{z,x}}_{T}) \cdot \grad \xi (S^{0,\zeta_{z,x}}_{T}) \right] + \varphi_j^\rem(T;0,\zeta_{z,x}) \right) \,.
	\end{align*}
	Then 
	\begin{align*}
		\ip{w(0), q_0} & = \int \sum_{j=1}^d (z_j - x_j) \E \left[ (\p_{x_j} S^{0,\zeta_{z,x}}_{T}) \cdot \grad \xi (S^{0,\zeta_{z,x}}_{T}) \right] \nu(dx) \\
		& \qquad + \int \sum_{j=1}^d (z_j - x_j) \varphi_j^\rem(T;0,\zeta_{z,x}) \nu(dx) \,.
	\end{align*}
	Since $z \in B(c_0,2r_\cut)$ and $\supp(\nu) \subset B(c_0,2r_\cut)$, we know all $z_j - x_j$ and $\zeta_{z,x}$ are always bounded.
	Thus~\eqref{e:exp-decay-general} follows, finishing the proof. 
\end{proof}

\begin{proof}[Proof of Lemma~\ref{lm:connect-r1-q}]
	Notice $(q_t)_{t \ge 0}$ can be interpreted as 
	\begin{align*}
		\frac{d}{dt} \ip{f, q_t} & = \int \left( \frac{\sigma^2}{2} \abs{x-M(\mu_t)}^2 \phi(x)^2 \lap f(x) - \lambda (x-M(\mu_t)) \cdot \grad f(x) \right) q_t(dx) \\
		& \quad + \frac{ \ip{\lambda \grad f - \sigma^2 (\cdot - M(\mu_t)) \phi^2 \lap f, \mu_t} }{\ip{e^{-\alpha\efn}, \mu_t}} \cdot \int (x-M(\mu_t)) e^{-\alpha \efn(x)} q_t(dx) \\
		& \quad + \int f(x) r_t(dx) \,.
	\end{align*}
	Comparing $L_{\mu_t}$ and $L_{\delta_{\tilde x^\mu}}$ gives
	\begin{align*}
		\ip{f, r^1_t} & = \int \frac{\sigma^2}{2} (\tilde x^\mu - M(\mu_t))^\top (2x- \tilde x^\mu - M(\mu_t)) \phi(x)^2 \lap f(x) q_t(dx) \\
        & \qquad + \int \lambda (M(\mu_t) - \tilde x^\mu)^\top \grad f(x) q_t(dx) \\
		& \quad + \frac{ \ip{\lambda \grad f - \sigma^2 (\cdot - M(\mu_t)) \phi^2 \lap f, \mu_t} }{\ip{e^{-\alpha\efn}, \mu_t}} \cdot (\tilde x^\mu - M(\mu_t)) \ip{e^{-\alpha\efn}, q_t} \\
		& \quad + \left( \frac{ \ip{\lambda \grad f - \sigma^2 (\cdot - M(\mu_t)) \phi^2 \lap f, \mu_t} }{\ip{e^{-\alpha\efn}, \mu_t}} - \lambda e^{\alpha \efn(\tilde x^\mu)} \grad f(\tilde x^\mu) \right) \cdot \ip{(\cdot - \tilde x^\mu) e^{-\alpha\efn}, q_t} \,.
	\end{align*}
	Define $R^{1,i}_t \in (W^{1,\infty})'$ by 
	\begin{align*}
		\ip{g, R^{1,i}_t} & \defeq \int \frac{\sigma^2}{2} (\tilde x^\mu - M(\mu_t))^\top (2x- \tilde x^\mu - M(\mu_t)) \phi(x)^2 \p_{x_i} g(x) q_t(dx) \\
        & \qquad + \int \lambda (M(\mu_t)_i - \tilde x^\mu_i) g(x)  q_t(dx) \\
		& \quad + \frac{ \ip{\lambda (\tilde x^\mu_i - M(\mu_t)_i) g - \sigma^2 \phi^2 (\tilde x^\mu - M(\mu_t))^\top (\cdot - M(\mu_t)) \p_{x_i} g, \mu_t} }{\ip{e^{-\alpha\efn}, \mu_t}}  \ip{e^{-\alpha\efn}, q_t} \\
		& \quad + \left( \frac{ \ip{\lambda g, \mu_t} }{\ip{e^{-\alpha\efn}, \mu_t}} - \lambda e^{\alpha \efn(\tilde x^\mu)} g(\tilde x^\mu) \right)  \ip{(\cdot - \tilde x^\mu)_i e^{-\alpha\efn}, q_t} \\
		& \qquad - \frac{ \ip{\sigma^2 (\cdot - M(\mu_t)) \phi^2 \p_{x_i} g, \mu_t} }{\ip{e^{-\alpha\efn}, \mu_t}}  \cdot \ip{(\cdot - \tilde x^\mu) e^{-\alpha\efn}, q_t} \,.
	\end{align*}
	Then we have  
	\begin{equation*}
		\ip{f, r^1_t} = \sum_{i=1}^d \ip{\p_{x_i} f, R^{1,i}_t} \,.
	\end{equation*}

	Now we look at the relation between $R^{1,i}_t$'s and $q_t$. 
	Recall from Proposition~\ref{pp:CBO-mf-decay} that 
	\begin{equation}
		\label{e:m-mu-t-decay}
		\abs{M(\mu_t) - \tilde x^\mu} \le C e^{-\kappa t} \,, \qquad \int \abs{x - M(\mu_t)} \mu_t(dx) \le C e^{-\kappa t} 
	\end{equation}
	for all $t \ge 0$.

	Fix an arbitrary $i \in [d]$, and we let $\ip{g, R^{1,i}_t} = P_1 + P_2 + P_3 + P_4 + P_5$ in the order as defined above. 
	Notice that $P_1$ and $P_2$ are both linear in $\tilde x^\mu - M(\mu_t)$, so 
	\begin{align*}
		\abs{P_1 + P_2} & \le C_1(\lambda,\sigma,\phi) C e^{-\kappa t} \norm{g}_{n+1,\infty} \norm{q_t}_{(n,\infty)'} \,,
	\end{align*}
	where $C_1(\lambda,\sigma,\phi)$ is some constant depending linearly on $\lambda, \sigma^2$, and $\norm{\phi}_{n,\infty}$.

	Observe that $e^{-\alpha \efn}$ and $x e^{-\alpha \efn}$ are both $W^{n,\infty}$.
    In addition, by Jensen's inequality, 
    \begin{equation*}
        \ip{e^{-\alpha \efn}, \mu_t} \ge \exp(-\alpha \ip{\efn, \mu_t}) \ge \exp(-\alpha c_\efn 9 r_\cut^2). 
    \end{equation*}
	Then similar to the above terms, we have
	\begin{equation*}
		\abs{P_3 + P_5} \le C_3(\lambda, \sigma, \phi, e^{\alpha c_\efn 9 r_\cut^2}) C e^{-\kappa t} \norm{g}_{1,\infty} \norm{q_t}_{(n,\infty)'} \,.
	\end{equation*}

	For $P_4$, we see that 
	\begin{equation*}
		\frac{\ip{\lambda g, \mu_t} }{\ip{e^{-\alpha\efn}, \mu_t}} - \lambda e^{\alpha \efn(\tilde x^\mu)} g(\tilde x^\mu) = \lambda \frac{\ip{g, \mu_t - \delta_{\tilde x^\mu}} }{\ip{e^{-\alpha\efn}, \mu_t}} +  \lambda e^{\alpha \efn(\tilde x^\mu)} g(\tilde x^\mu)  \frac{ \ip{e^{-\alpha\efn}, \delta_{\tilde x^\mu} - \mu_t} }{\ip{e^{-\alpha\efn},\mu_t}} \,.
	\end{equation*}
	Recall that Proposition~\ref{pp:CBO-mf-decay} implies $\norm{\mu_t - \delta_{\tilde x}}_{(1,\infty)'} \le C e^{-\kappa t}$, so 
	\begin{equation*}
		\abs{P_4} \le C_4 C e^{-\kappa t} \norm{g}_{1,\infty} \norm{q_t}_{(n,\infty)'} \,.
	\end{equation*}
	Summarizing the above, we get 
	\begin{equation*}
		\abs{\ip{g, R^{1,i}_t}} \le \bar C e^{-\kappa t} \norm{g}_{n+1,\infty} \norm{q_t}_{(n,\infty)'} \,,
	\end{equation*}
	where $\bar C$ depends only on $\lambda, \sigma, \alpha$, $\phi$, and some finite parameters of $\efn$. 
	
	Thus we conclude that 
	\begin{equation*}
		\sup_{i \in [d]} \norm{R^{1,i}_t}_{(n+1,\infty)'} \le \bar C e^{-\kappa t} \norm{q_t}_{(n,\infty)'} \,.
	\end{equation*}
\end{proof}

\begin{proof}[Proof of Lemma~\ref{lm:bound-q}]
	\restartsteps
	\step[Decomposition of norms]
	Fix arbitrary $T > 0$. 
	Recall that 
	\begin{equation*}
		\norm{q_T}_{(n \lor (1+n_0),\infty)'} = \sup_{\norm{\xi}_{n \lor (1+n_0), \infty}} \ip{\xi, q_T} \,.
	\end{equation*}
	For arbitrary $\xi \in W^{n \lor (1+n_0), \infty}$, we consider the backward Cauchy problem 
	\begin{equation*}
		\begin{cases}
			\p_t w_t + \lin_{\mu_t} w_t + \a_{\mu_t} w_t = 0 \,, & t \in [0,T) \,, \\
			w_T = \xi \,.
		\end{cases}
	\end{equation*}
	Then 
	\begin{equation*}
		\ip{\xi, q_T} = \ip{w(0), q_0} + \int_0^T \ip{w(t), r_t} dt = \ip{w(0), q_0} - \int_0^T \sum_{i=1}^d \ip{\p_{x_i} w(t), R^i_t} dt \,.
	\end{equation*}
	It remains to study the bounds for $w$. 

	\step[Recurrence of $w$]
	Fix $T, \xi, \mu$, we define 
	\begin{equation*}
		L_t = L^{T,\xi,\mu}_t \defeq \ip{\lambda \grad w(t) - \sigma^2 (\cdot-M(\mu_t)) \phi^2 \lap w(t), \mu_t } \in \R^d
	\end{equation*}
	for $t \in [0,T]$. 
	Note that $L$ does not depend on $x$. 
	The backward equation for $w$ can be written more explicitly as 
	\begin{equation*}
		\p_t w(t,x) + \lin_{\mu_t} w(t,x) + \frac{L_t}{\ip{e^{-\alpha\efn},\mu_t}} \cdot G_t(x) = 0 \,,
	\end{equation*}
	where 
	\begin{equation*}
		G_t(x) = G^\mu_t(x) \defeq (x-M(\mu_t)) e^{-\alpha \efn(x)} \,.
	\end{equation*}
	Applying the Feynman-Kac formula, we have an (implicit) formula for $w$,
	\begin{equation*}
		w(t,x) = \E [\xi(Y^{t,x,\mu}_T)] + \int_t^T \frac{L_u}{\ip{e^{-\alpha\efn},\mu_u}} \cdot \E [G^\mu_u(Y^{t,x,\mu}_u)] du \,,
	\end{equation*}
	where $(Y^{t,x,\mu}_u)_{u \in [t,T]}$ satisfies the SDE 
	\begin{equation*}
		d Y_u = -\lambda (Y_u - M(\mu_u)) du + \sigma \abs{Y_u - M(\mu_u)} \phi(Y_u) dW_u \,, \quad u \in [t,T] \,,
	\end{equation*}
	with $Y^{t,x,\mu}_t = x$. 
    Similar to the situation in Lemma~\ref{lm:erg-at-delta}, $w_t \in W^{n,\infty}(B(c_0, 2r_\cut))$.
	Taking derivatives in $x$, we get 
	\begin{equation*}
		\p_{x_i} w(t,x) = \E [D\xi(Y^{t,x,\mu}_T) (\p_{x_i} Y^{t,x,\mu}_T)] + \int_t^T \E [(\p_{x_i} Y^{t,x,\mu}_u)^\top D G_u(Y^{t,x,\mu}_u)] \frac{L_u}{\ip{e^{-\alpha\efn}, \mu_u}} du 
	\end{equation*}
	and 
	\begin{align*}
		& \p_{x_i}^2 w(t,x) = \E [D\xi(Y^{t,x,\mu}_T) (\p_{x_i}^2 Y^{t,x,\mu}_T) + (\p_{x_i} Y^{t,x,\mu}_T)^\top D^2 \xi(Y^{t,x,\mu}_T) (\p_{x_i} Y^{t,x,\mu}_T)] \\
		& \quad + \int_t^T \E [ (\p_{x_i}^2 Y^{t,x,\mu}_u)^\top D G_u(Y^{t,x,\mu}_u) + (\p_{x_i} Y^{t,x,\mu}_u)^\top D^2 G_u(Y^{t,x,\mu}_u) (\p_{x_i} Y^{t,x,\mu}_u) ] \frac{L_u}{\ip{e^{-\alpha\efn}, \mu_u}} du \,,
	\end{align*}
	where $Df$ and $D^2f$ are the Jacobian matrix and Hessian matrix of $f$, respectively.
	Note that $D^2 G_u$ is a 3-tensor, and $z^\top D^2 G_u(z) z \in \R^{1 \times d}$ with 
	\begin{equation*}
		\paren[\big]{z^\top D^2 G_u(z) z}_j = \sum_{k,k'=1}^d z_k z_{k'} \p_{k,k'} G^j_u(z)
	\end{equation*}
	for each $z \in \R^d$ and $j \in [d]$.
	It suffices to find a uniform uppper bound on $\abs{L_t}_1$, the Euclidean 1-norm.

	Let $(X_t)_{t \in [0,T]}$ be some process with law $(\mu_t)_{t \in [0,T]}$. 
	It satisfies exactly the CBO dynamics $dX_t = -\lambda (X_t - M(\mu_t)) dt + \sigma \abs{X_t - M(\mu_t)} \phi(X_t) dW_t$.
	This means $(Y^{t,x,\mu}_u)_{u \in [t,T]} \sim (X_u)_{u \in [t,T]}$ if $x \sim \mu_t$. 
	So by Markov property, for each $i \in [d]$ we have 
	\begin{align*}
		L^i_t & = \lambda \E [D\xi(X_T) (\p_{x_i} Y^{t,\cdot,\mu}_T (X_t))] - \sigma^2 \sum_{j=1}^d \E \left[(X_t - M(\mu_t))_i \phi(X_t)^2 \right. \\
		& \qquad \qquad  \left. \paren[\big]{ D\xi(X_T) (\p_{x_j}^2 Y^{t,\cdot,\mu}_T (X_t)) + (\p_{x_j} Y^{t,\cdot,\mu}_T (X_t))^\top D^2 \xi(X_T) (\p_{x_j} Y^{t,\cdot,\mu}_T(X_t)) } \right] \\
		& \quad + \int_t^T B^i_{u,t} L_u  du
	\end{align*}
	for some appropriate coupling of $X$ and $Y$, where $B^i_{u,t} \in \R^{1 \times d}$ is defined by 
	\begin{align*}
		B^i_{u,t} & \defeq \frac{\lambda}{\ip{e^{-\alpha\efn},\mu_u}} \E [(\p_{x_i} Y^{t,\cdot,\mu}_u (X_t))^\top DG(X_u)] - \frac{\sigma^2}{\ip{e^{-\alpha\efn},\mu_u}} \sum_{j=1}^d \E \left[ (X_t - M(\mu_t))_i \phi(X_t)^2 \right. \\
		& \qquad \left. \paren[\big]{ (\p_{x_j}^2 Y^{t,\cdot,\mu}_u (X_t))^\top DG_u(X_u) + (\p_{x_j} Y^{t,\cdot,\mu}_u (X_t))^\top D^2 G_u(X_u) (\p_{x_j} Y^{t,\cdot,\mu}_u(X_t)) } \right] \,.
	\end{align*}
	We will show in the next step that $\abs{B^{i,k}_{u,t}} \le e^{-\lambda (u-t)}(\lambda \one{i=k} + C_0 e^{-\kappa_1 u})$ for some $C_0, \kappa_1 > 0$, and then 
	\begin{equation*}
		\abs{L_t}_1 \le C  e^{-\lambda (T-t)} \norm{\xi}_{2,\infty} + \int_t^T e^{-\lambda(u-t)} (\lambda + C_0 d e^{-\kappa_1 u}) \abs{L_u}_1 du \,.
	\end{equation*}

	\step[Upper bounds of $B_{u,t}$]
	For each $k \in [d]$, we may split the $k$-th component of $B^i_{u,t}$ by 
	\begin{align*}
		B^{i,k}_{u,t} & = \lambda e^{-\lambda (u-t)} + \left( \frac{\lambda}{\ip{e^{-\alpha\efn},\mu_u}} \E[\p_{x_i} Y^{t,\cdot,\mu}_u(X_t)] \cdot e^{-\alpha \efn(\tilde x^\mu)} \ve_k - \lambda e^{-\lambda (u-t)} \right) \\
		& + \frac{\lambda}{\ip{e^{-\alpha\efn},\mu_u}} \E \left[ (\p_{x_i} Y^{t,\cdot,\mu}_u(X_t)) \cdot (\p_k G_u(X_u) - e^{-\alpha \efn(\tilde x^\mu)} \ve_k) \right] \\
		& - \frac{\sigma^2}{\ip{e^{-\alpha\efn},\mu_u}} \sum_{j=1}^d \E \left[(X_t - M(\mu_t))_i \phi(X_t)^2  (\p_{x_j}^2 Y^{t,\cdot,\mu}_u(X_t)) \cdot \p_k G_u(X_u) \right] \\
		& - \frac{\sigma^2}{\ip{e^{-\alpha\efn},\mu_u}} \sum_{j=1}^d \E \left[(X_t - M(\mu_t))_i \phi(X_t)^2  (\p_{x_j} Y^{t,\cdot,\mu}_u (X_t))^\top D^2 G^k_u(X_u) (\p_{x_j} Y^{t,\cdot,\mu}_u(X_t))  \right] \\
		& \defeq \lambda e^{-\lambda (u-t)} + P^k_1 + P^k_2 + P^k_3 + P^k_4 \,,
	\end{align*}
	where $\ve_k$ is the $k$-th vector in the canonical basis of $\R^d$.
	Observe that 
	\begin{equation*}
		P^i_1 = \lambda e^{-\lambda (u-t)} \left( \frac{\ip{e^{-\alpha\efn},\delta_{\tilde x^\mu}}}{\ip{e^{-\alpha\efn},\mu_u}} - 1 \right)
	\end{equation*}
	due to Lemma~\ref{lm:prop-y}. 
	Since $\ip{e^{-\alpha\efn},\mu_u}$ is bounded below, and $\norm{\mu_u - \delta_{\tilde x^\mu}} \le C e^{-\kappa u}$, we have 
	\begin{equation*}
		\abs{P^i_1} \le C_1 \lambda e^{-\lambda (u-t)} e^{-\kappa u} \,,
	\end{equation*}
    and at the same time 
    \begin{equation*}
        \lambda e^{-\lambda (u-t)} + P^k_1 = 0 
    \end{equation*}
    for $k \neq i$.

	Notice that 
	\begin{equation*}
		\p_k G_u(x) =  e^{-\alpha \efn(x)} ( \ve_k -\alpha \p_k \efn(x) (x - M(\mu_u))) \,.
	\end{equation*}
	Then, as the 2nd-order derivatives of $\efn$ are bounded, we have
	\begin{align*}
		\E \abs{\p_k G_u(X_u) - e^{-\alpha \efn(\tilde x^\mu)}\ve_k}^2 & \le 2 \E \paren{e^{-\alpha\efn(X_u)} - e^{-\alpha \efn(\tilde x^\mu)}}^2 + 2\alpha \E \abs{\p_k \efn(X_u) (X_u - M(\mu_u))}^2 \\ 
		& \lesssim O(e^{-2\kappa u}) \,.
	\end{align*}
	By Lemma~\ref{lm:prop-y} and Cauchy-Schwarz inequality, 
	\begin{equation*}
		\abs{P^k_2} \le C_2 e^{-(\lambda - \frac{\sigma^2 \kappa_2}{2}) (u-t)} e^{-\kappa u} \le C_2 e^{-\lambda (u-t)} e^{-(\kappa - \frac{\sigma^2 \kappa_2}{2}) u} \,.
	\end{equation*}
	Recall from Proposition~\ref{pp:CBO-mf-decay} that $\kappa \sim \lambda - d \sigma^2$.
	So $\kappa - \frac{\sigma^2 \kappa_2}{2} > 0$ when $\lambda$ is large enough.

	Further, 
	\begin{align*}
		\p_{\ell,\ell'} G_u(x) & = e^{-\alpha \efn(x)} \left(-\alpha \p_{\ell'} \efn(x) \ve_\ell - \alpha \p_\ell \efn(x) \ve_{\ell'} \right. \\
		& \qquad \left. +\alpha^2 \p_\ell \efn(x) \p_{\ell'} \efn(x) (x-M(\mu_u)) - \alpha \p_{\ell,\ell'} \efn(x) (x-M(\mu_u)) \right) \,.
	\end{align*}
	So we have 
	\begin{align*}
		& (\p_{x_j} Y^{t,\cdot,\mu}_u (X_t))^\top D^2 G^k_u(X_u) (\p_{x_j} Y^{t,\cdot,\mu}_u(X_t)) \\
		& = \sum_{\ell,\ell'=1}^d (\p_{x_j} Y^{t,\cdot,\mu}_{u,\ell} (X_t)) (\p_{x_j} Y^{t,\cdot,\mu}_{u,\ell'} (X_t)) \p_{\ell,\ell'} G^k_u(X_u)\,.
	\end{align*}
	Since $\abs{X_u} \le \abs{c_0} + 2r_\cut < \infty$, there exists some constant $K_2$ such that $\abs{\p_{\ell,\ell'} G^k_u(X_u)} \le K_2$ a.s.
	This gives 
	\begin{align*}
		\abs{P^k_4} & \le \frac{\sigma^2}{\ip{e^{-\alpha\efn},\mu_u}} \sum_{j,\ell,\ell'=1}^d \E \left[ \abs{X_t - M(\mu_t)} \abs{\p_{x_j} Y^{t,\cdot,\mu}_{u,\ell} (X_t)} \abs{\p_{x_j} Y^{t,\cdot,\mu}_{u,\ell'} (X_t)} \right] \\
		& = \frac{\sigma^2}{\ip{e^{-\alpha\efn},\mu_u}} \sum_{j=1}^d \E \left[ \abs{X_t - M(\mu_t)} \left( \sum_{\ell=1}^d \abs{\p_{x_j} Y^{t,X_t,\mu}_{u,\ell}} \right)^2 \right] \\
		& \le \frac{\sigma^2}{\ip{e^{-\alpha\efn},\mu_u}} d \sum_{j=1}^d \E \left[ \abs{X_t-M(\mu_t)} \sum_{\ell=1}^d \abs{\p_{x_j} Y^{t,X_t,\mu}_{u,\ell}}^2 \right] \,.
	\end{align*}
	Applying Cauchy-Schwarz, and using Proposition~\ref{pp:CBO-mf-decay} and Lemma~\ref{lm:prop-y} to bound the two components respectively, we obtain 
	\begin{equation*}
		\abs{P^k_4} \le C_4 d^3 e^{-\kappa t} e^{-2(\lambda - \frac{\sigma^2 \kappa_4}{4}) (u-t)} \le C_4 d^3 e^{-\lambda(u-t)} e^{-((\lambda-\frac{\sigma^2 \kappa_4}{4}) \land \kappa) u} \,. 
	\end{equation*}

	For $P_3$, we apply the estimate of $\p_k G_u(X_u)$ in the analysis of $P_2$,
	\begin{equation*}
		\E \abs{\p_k G_u(X_u) -  e^{-\alpha \efn(\tilde x^\mu)}\ve_k}^2 \le O(e^{-2\kappa u}) \,,
	\end{equation*} 
	so that, with Lemma~\ref{lm:prop-y} again,
	\begin{equation*}
		\E \abs{(X_t - M(\mu_t))_i \phi(X_t)^2  (\p_{x_j}^2 Y^{t,\cdot,\mu}_u(X_t)) \cdot (\p_k G_u(X_u) -  e^{-\alpha \efn(\tilde x^\mu)}\ve_k)} \le C_3 e^{-\lambda (u-t)} e^{-(\kappa - \frac{\sigma^2 \kappa_2}{2}) u} \,.
	\end{equation*}
	In addition, using tower property and then Lemma~\ref{lm:prop-y} for the inner layer, we have 
	\begin{align*}
		& \E \left[ (X_t - M(\mu_t))_i \phi(X_t)^2  (\p_{x_j}^2 Y^{t,\cdot,\mu}_{u,k}(X_t)) \right] \\
		& = \E \left[ (X_t - M(\mu_t))_i \phi(X_t)^2  \E[ (\p_{x_j}^2 Y^{t,\cdot,\mu}_{u,k}(X_t)) \mid X_t] \right] = 0 \,.
	\end{align*}
	Thus 
	\begin{equation*}
		\abs{P^k_3} \le C_3 d e^{-\lambda(u-t)} e^{-(\kappa - \frac{\sigma^2 \kappa_2}{2}) u} \,.
	\end{equation*}

	Take the maximum possible $\kappa_1 > 0$. 
	Summarizing the above four terms for all $k$, we get 
	\begin{equation*}
		\abs{B^{i,k}_{u,t}} \le e^{-\lambda(u-t)} (\lambda \one{i=k} + C_0 e^{-\kappa_1 u}) \,.
	\end{equation*}

	\step[Grönwall's inequality]
	It remains to deal with the initial term 
	\begin{align*}
		& \lambda \E [D\xi(X_T) (\p_{x_i} Y^{t,\cdot,\mu}_T (X_t))] - \sigma^2 \sum_{j=1}^d \E \left[(X_t - M(\mu_t))_i \phi(X_t)^2 \right. \\
		& \qquad \left. \paren[\big]{ D\xi(X_T) (\p_{x_j}^2 Y^{t,\cdot,\mu}_T (X_t)) + (\p_{x_j} Y^{t,\cdot,\mu}_T (X_t))^\top D^2 \xi(X_T) (\p_{x_j} Y^{t,\cdot,\mu}_T(X_t)) } \right] \,.
	\end{align*}
	We apply the same techniques to see that it is upper bounded by 
	\begin{equation*}
		C e^{-\lambda (T-t)} (\norm{\grad \xi}_\infty + \norm{\grad^2 \xi}_\infty) \,.
	\end{equation*}
	Thus 
	\begin{equation*}
		\abs{L^i_t} \le C  e^{-\lambda (T-t)} \norm{\xi}_{2,\infty} + \int_t^T e^{-\lambda(u-t)} (\lambda \abs{L^i_u} + C_0 e^{-\kappa_1 u} \abs{L_u}_1) du \,,
	\end{equation*}
	and consequently 
	\begin{equation*}
		e^{-\lambda t} \abs{L_t}_1 \le C  e^{-\lambda T} \norm{\xi}_{2,\infty} + \int_t^T (\lambda + C_0 d e^{-\kappa_1 u}) e^{-\lambda u} \abs{L_u}_1 du \,.
	\end{equation*}
	Apply Grönwall's inequality backward in time, we obtain 
	\begin{equation*}
		e^{-\lambda t} \abs{L_t}_1 \le C e^{-\lambda T} \norm{\xi}_{2,\infty} \exp \left( \int_t^T (\lambda + C_0 d e^{-\kappa_1 u}) du \right) \,,
	\end{equation*}
	so that 
	\begin{equation*}
		\abs{L_t}_1 \le C e^{C_0 e^{-\kappa_1 (T-t)}} \norm{\xi}_{2,\infty} \,.
	\end{equation*}

	\step[Summary]
	Now that the term $\frac{L^{T,\xi,\mu}_u}{\ip{e^{-\alpha\efn},\mu_u}}$ is uniformly bounded, we apply an extension of Lemma~\ref{lm:prop-y} to see that 
	\begin{equation*}
		\norm{\p_{x_i} w(t)}_{n_0,\infty} \le c e^{-\lambda (T-t)/2} \norm{\xi}_{1+n_0,\infty} + \int_t^T C e^{-\lambda (u-t)/2} \norm{\xi}_{2,\infty} du \le \bar C \norm{\xi}_{1+n_0,\infty} \,.
	\end{equation*}
	Also, with the decay of $R_t$, we get 
	\begin{equation*}
		\abs{ \int_0^T \sum_{i=1}^d \ip{\p_{x_i} w(t), R^i_t} dt } \le \frac{\bar C d K}{\lambda_0} \,.
	\end{equation*}
	For specific $q_0$, we have $\ip{w(0), q_0}$ equals 
	\begin{equation*}
		\p_{x_j} w(0,z) \,, \qquad \p_{x_j}^2 w(0,z) \,, \qquad \int \grad w(0,\zeta_{z,x}) \cdot (z-x) \nu(dx)
	\end{equation*}
	respectively, the bound of which follows in the same way as above. 
	Therefore 
	\begin{equation*}
		\norm{q_T}_{(n \lor (1+n_0),\infty)'} \lesssim O(1+K) \,.
	\end{equation*}
\end{proof}

\bibliographystyle{halpha-abbrv}
\bibliography{refs.bib}

\begin{thebibliography}{PTTM17}
\expandafter\ifx\csname url\endcsname\relax
  \def\url#1{\texttt{#1}}\fi
\expandafter\ifx\csname doi\endcsname\relax
  \def\doi#1{\burlalt{doi:#1}{http://dx.doi.org/#1}}\fi
\expandafter\ifx\csname urlprefix\endcsname\relax\def\urlprefix{URL }\fi
\expandafter\ifx\csname href\endcsname\relax
  \def\href#1#2{#2}\fi
\expandafter\ifx\csname burlalt\endcsname\relax
  \def\burlalt#1#2{\href{#2}{#1}}\fi

\bibitem[AK91]{AartsKorst1991}
E.~Aarts and J.~Korst.
\newblock {\em Simulated Annealing and Boltzmann Machines: A Stochastic
  Approach to Combinatorial Optimization and Neural Computing}.
\newblock Wiley Series in Discrete Mathematics \& Optimization. John Wiley and
  Sons Inc., 1991.

\bibitem[BFM97]{BaeckFogelMichalewicz1997}
T.~Baeck, D.~B. Fogel, and Z.~Michalewicz, editors.
\newblock {\em Handbook of Evolutionary Computation}.
\newblock CRC Press, 1st edition, 1997.
\newblock \doi{10.1201/9780367802486}.

\bibitem[BLPR17]{BuckdahnLiPengRainer2017}
R.~Buckdahn, J.~Li, S.~Peng, and C.~Rainer.
\newblock Mean-field stochastic differential equations and associated pdes.
\newblock {\em The Annals of Probability}, 45(2):824--878, 2017.

\bibitem[CCTT18]{CarriloChoiTotzeckTse2018}
J.~A. Carrillo, Y.-P. Choi, C.~Totzeck, and O.~Tse.
\newblock An analytical framework for consensus-based global optimization
  method.
\newblock {\em Mathematical Models and Methods in Applied Sciences},
  28(06):1037--1066, 2018,
  \burlalt{https://doi.org/10.1142/S0218202518500276}{http://arxiv.org/abs/https://doi.org/10.1142/S0218202518500276}.
\newblock \doi{10.1142/S0218202518500276}.

\bibitem[CD18a]{CarmonaDelarue2018I}
R.~Carmona and F.~Delarue.
\newblock {\em Probabilistic Theory of Mean Field Games with Applications I.
  Mean Field FBSDEs, Control, and Games}.
\newblock Probability Theory and Stochastic Modelling. Springer Cham, New York,
  NY, 2018.
\newblock \doi{10.1007/978-3-319-58920-6}.

\bibitem[CD18b]{CarmonaDelarue2018II}
R.~Carmona and F.~Delarue.
\newblock {\em Probabilistic Theory of Mean Field Games with Applications II.
  Mean Field Games with Common Noise and Master Equations}.
\newblock Probability Theory and Stochastic Modelling. Springer Cham, New York,
  NY, 2018.
\newblock \doi{10.1007/978-3-319-56436-4}.

\bibitem[CDLL19]{CardaliaguetDelarueLasryLions2019}
P.~Cardaliaguet, F.~Delarue, J.-M. Lasry, and P.-L. Lions.
\newblock {\em The Master Equation and the Convergence Problem in Mean Field
  Games}.
\newblock Annals of Mathematics Studies. Princeton University Press, Princeton,
  NJ, 2019.

\bibitem[CJZZ24]{CarrilloJinZhangZhu2024}
J.~A. Carrillo, S.~Jin, H.~Zhang, and Y.~Zhu.
\newblock An interacting particle consensus method for constrained global
  optimization, 2024.
\newblock \doi{10.48550/arXiv.2405.00891}.

\bibitem[Cor24]{Cormier2024}
Q.~Cormier.
\newblock On the stability of the invariant probability measures of
  mckean-vlasov equations, 2024.
\newblock \doi{10.48550/arXiv.2201.11612}.

\bibitem[CST22]{ChassagneuxSzpruchTse2022}
J.-F. Chassagneux, L.~Szpruch, and A.~Tse.
\newblock {Weak quantitative propagation of chaos via differential calculus on
  the space of measures}.
\newblock {\em The Annals of Applied Probability}, 32(3):1929 -- 1969, 2022.
\newblock \doi{10.1214/21-AAP1725}.

\bibitem[DT25]{DelarueTse2021}
F.~Delarue and A.~Tse.
\newblock Uniform in time weak propagation of chaos on the torus.
\newblock {\em To appear in Annales de l’Institut {H}enri {P}oincar\'e},
  2025+.

\bibitem[FKR24]{FornasierKlockRiedl2021}
M.~Fornasier, T.~Klock, and K.~Riedl.
\newblock Consensus-based optimization methods converge globally.
\newblock {\em SIAM Journal on Optimization}, 34(3):2973--3004, 2024.
\newblock \doi{10.1137/22M1527805}.

\bibitem[Fog06]{Fogel2006}
D.~B. Fogel.
\newblock {\em Evolutionary Computation: Toward a New Philosophy of Machine
  Intelligence, 3rd Edition}.
\newblock IEEE Press Series on Computational Intelligence. Wiley-IEEE Press,
  2006.

\bibitem[GHV23]{GerberHoffmannVaes2023}
N.~J. Gerber, F.~Hoffmann, and U.~Vaes.
\newblock Mean-field limits for consensus-based optimization and sampling,
  2023.
\newblock \doi{10.48550/arXiv.2312.07373}.

\bibitem[GLBM23]{GuillinLebrisMonmarche2023}
A.~Guillin, P.~Le~Bris, and P.~Monmarché.
\newblock On systems of particles in singular repulsive interaction in
  dimension one: log and {Riesz} gas.
\newblock {\em Journal de l'École polytechnique {\textemdash} Mathématiques},
  10:867--916, 2023.
\newblock \doi{10.5802/jep.235}.

\bibitem[Hol92]{Holland1992}
J.~H. Holland.
\newblock {\em Adaptation in Natural and Artificial Systems: An Introductory
  Analysis with Applications to Biology, Control, and Artificial Intelligence}.
\newblock Complex Adaptive Systems. MIT Press, Cambridge, MA, 1992.
\newblock \doi{10.7551/mitpress/1090.001.0001}.

\bibitem[HQ22]{HuangQiu2022}
H.~Huang and J.~Qiu.
\newblock On the mean-field limit for the consensus-based optimization.
\newblock {\em Mathematical Methods in the Applied Sciences},
  45(12):7814--7831, 2022,
  \burlalt{https://onlinelibrary.wiley.com/doi/pdf/10.1002/mma.8279}{http://arxiv.org/abs/https://onlinelibrary.wiley.com/doi/pdf/10.1002/mma.8279}.
\newblock \doi{https://doi.org/10.1002/mma.8279}.

\bibitem[HQR24]{HuangQiuRiedl2024}
H.~Huang, J.~Qiu, and K.~Riedl.
\newblock Consensus-based optimization for saddle point problems.
\newblock {\em SIAM Journal on Control and Optimization}, 62(2):1093--1121,
  2024,
  \burlalt{https://doi.org/10.1137/22M1543367}{http://arxiv.org/abs/https://doi.org/10.1137/22M1543367}.
\newblock \doi{10.1137/22M1543367}.

\bibitem[KE95]{KennedyEberhart1995}
J.~Kennedy and R.~Eberhart.
\newblock Particle swarm optimization.
\newblock {\em Proceedings of ICNN'95 - International Conference on Neural
  Networks}, 4:1942--1948, 1995.
\newblock \doi{10.1109/ICNN.1995.488968}.

\bibitem[Lac21]{Lacker2021}
D.~Lacker.
\newblock Hierarchies, entropy, and quantitative propagation of chaos for mean
  field diffusions.
\newblock {\em Probab. Math. Phys.}, 4, 05 2021.
\newblock \doi{10.2140/pmp.2023.4.377}.

\bibitem[LLF23]{LackerFlem2023}
D.~Lacker and L.~Le~Flem.
\newblock Sharp uniform-in-time propagation of chaos.
\newblock {\em Probab. Theory Related Fields}, 187(1):443--480, 2023.
\newblock \doi{10.1007/s00440-023-01192-x}.

\bibitem[Mal01]{Malrieu2001}
F.~Malrieu.
\newblock Logarithmic sobolev inequalities for some nonlinear pde's.
\newblock {\em Stochastic Processes and their Applications}, 95(1):109--132,
  2001.
\newblock \doi{https://doi.org/10.1016/S0304-4149(01)00095-3}.

\bibitem[Mal03]{Malrieu2003}
F.~Malrieu.
\newblock {Convergence to equilibrium for granular media equations and their
  Euler schemes}.
\newblock {\em The Annals of Applied Probability}, 13(2):540 -- 560, 2003.
\newblock \doi{10.1214/aoap/1050689593}.

\bibitem[MMW15]{MischlerMouhotWennberg2015}
S.~Mischler, C.~Mouhot, and B.~Wennberg.
\newblock A new approach to quantitative propagation of chaos for drift,
  diffusion and jump processes.
\newblock {\em Probability Theory and Related Fields}, 161(1):1--59, 2015.
\newblock \doi{10.1007/s00440-013-0542-8}.

\bibitem[PTTM17]{PinnauTotzeckTseMartin2017}
R.~Pinnau, C.~Totzeck, O.~Tse, and S.~Martin.
\newblock A consensus-based model for global optimization and its mean-field
  limit.
\newblock {\em Mathematical Models and Methods in Applied Sciences},
  27(01):183--204, 2017.
\newblock \doi{10.1142/S0218202517400061}.

\bibitem[RR02]{ReevesRowe2002}
C.~R. Reeves and J.~E. Rowe.
\newblock {\em Genetic Algorithms: Principles and Perspectives}.
\newblock Operations Research/Computer Science Interfaces Series. Springer New
  York, NY, 2002.
\newblock \doi{10.1007/b101880}.

\bibitem[RS23]{RosenzweigSerfaty2023}
M.~Rosenzweig and S.~Serfaty.
\newblock {Global-in-time mean-field convergence for singular Riesz-type
  diffusive flows}.
\newblock {\em Ann.Appl.Probab.}, 33(2):954 -- 998, 2023.
\newblock \doi{10.1214/22-AAP1833}.

\bibitem[SE98]{ShiEberhart1998}
Y.~Shi and R.~Eberhart.
\newblock A modified particle swarm optimizer.
\newblock {\em 1998 IEEE International Conference on Evolutionary Computation
  Proceedings. IEEE World Congress on Computational Intelligence (Cat.
  No.98TH8360)}, pages 69--73, 1998.
\newblock \doi{10.1109/ICEC.1998.699146}.

\bibitem[Szn91]{Sznitman1991}
A.-S. Sznitman.
\newblock {\em Topics in propagation of chaos}.
\newblock Lecture Notes in Mathematics. Springer-Verlag, New York, 1991.

\bibitem[Tse21]{Tse2021}
A.~Tse.
\newblock Higher order regularity of nonlinear fokker-planck pdes with respect
  to the measure component.
\newblock {\em Journal de Mathématiques Pures et Appliquées}, 150:134--180,
  2021.
\newblock \doi{https://doi.org/10.1016/j.matpur.2021.04.005}.

\end{thebibliography}

\appendix

\section{Details of proofs}
\label{s:complete-prf}

\begin{proof}[Complete proof of Lemma~\ref{l:m2-bound}]
	\restartsteps
	Fix $\mu \in \c_\cbo$ and $z_1, z_2 \in B(c_0, 2r_\cut)$.
	Recall that 
	\begin{equation*}
		\bar m^{(2)}(t;\mu,\delta_{z_1},\delta_{z_2}) = m^{(2)}(t;\mu,\delta_{z_1},\delta_{z_2}) - q^{(1)}_{0,\infty}(\mu,z_2) \cdot \grad m^{(1)}(t;\mu,\delta_{z_1}) \,,
	\end{equation*}
	where $m^{(2)}(\cdot;\mu,\delta_{z_1},\delta_{z_2}) = \cauchy{\mu}{\mu_0 - \delta_{z_2}}{(r_t)_{t \ge 0}}$ with 
	\begin{align*}
		\ip{f, r_t} & = \int \grad f(x) \cdot \frac{\delta b}{\delta m}(x,\mu_t)(q_2(t)) q_1(t)(dx) \\
		& + \int \grad f(x) \cdot \frac{\delta b}{\delta m}(x,\mu_t)(q_1(t)) q_2(t)(dx) \\
		& + \int \grad f(x) \cdot \frac{\delta^2 b}{\delta m^2}(x,\mu_t)(q_1(t),q_2(t)) \mu_t(dx) \\
		& + \frac{1}{2} \int \tr[ \grad^2 f(x) \frac{\delta a}{\delta m}(x,\mu_t)(q_2(t)) ] q_1(t)(dx) \\
		& + \frac{1}{2} \int \tr[ \grad^2 f(x) \frac{\delta a}{\delta m}(x,\mu_t)(q_1(t)) ] q_2(t)(dx) \\
		& + \frac{1}{2} \int \tr[ \grad^2 f(x) \frac{\delta^2 a}{\delta m^2}(x,\mu_t)(q_1(t),q_2(t)) ] \mu_t(dx) \,,
	\end{align*}
	where
	\begin{align*}
		\frac{\delta^2 b}{\delta m^2} (x,\mu,y_1,y_2) & = \frac{-\lambda (y_1+y_2-2M(\mu)) e^{-\alpha\cE(y_1)-\alpha\cE(y_2)}}{\ip{e^{-\alpha\cE},\mu}^2} \,, \\
		\frac{\delta^2 a}{\delta m^2} (x,\mu,y_1,y_2) & = \frac{2\sigma^2 \phi(x)^2 e^{-\alpha\cE(y_1)-\alpha\cE(y_2)}}{\ip{e^{-\alpha\cE},\mu}^2} \left( (x-M(\mu))\cdot (y_1+y_2-2M(\mu)) \right. \\
		& \qquad \qquad \qquad \left. + (y_1-M(\mu)) \cdot (y_2-M(\mu)) \right) I_{d \times d} \,.
	\end{align*}
	Define $G_t = G^\mu_t \defeq (x \mapsto (x-M(\mu_t)) e^{-\alpha \efn(x)})$, the remainder term $r$ becomes explicitly 
	\begin{align*}
		\ip{f, r_t} & = \int (\lambda \grad f(x) - \sigma^2 \phi(x)^2 (x-M(\mu_t)) \lap f(x)) m^{(1)}(t;\mu,\delta_{z_1})(dx) \cdot \frac{\ip{G_t, m^{(1)}(t;\mu,\delta_{z_2})}}{\ip{e^{-\alpha\efn},\mu_t}} \\
		& + \int (\lambda \grad f(x) - \sigma^2 \phi(x)^2 (x-M(\mu_t)) \lap f(x)) m^{(1)}(t;\mu,\delta_{z_2})(dx) \cdot \frac{\ip{G_t, m^{(1)}(t;\mu,\delta_{z_1})}}{\ip{e^{-\alpha\efn},\mu_t}} \\
		& - \int (\lambda \grad f(x) - \sigma^2 \phi(x)^2 (x-M(\mu_t)) \lap f(x)) \mu_t(dx) \cdot \\ 
		& \qquad \left( \frac{ \ip{G_t, m^{(1)}(t;\mu,\delta_{z_1})} \ip{e^{-\alpha\efn}, m^{(1)}(t;\mu,\delta_{z_2})} }{ \ip{e^{-\alpha\efn},\mu_t}^2 } + \right. \\
		& \qquad \quad \left. \frac{ \ip{G_t, m^{(1)}(t;\mu,\delta_{z_2})} \ip{e^{-\alpha\efn}, m^{(1)}(t;\mu,\delta_{z_1})} }{ \ip{e^{-\alpha\efn},\mu_t}^2 } \right) \\
		& + \int \sigma^2 \phi(x)^2 \lap f(x) \mu_t(dx) \frac{ \ip{G_t, m^{(1)}(t;\mu,\delta_{z_1})} \cdot \ip{G_t, m^{(1)}(t;\mu,\delta_{z_2})} }{ \ip{e^{-\alpha\efn},\mu_t}^2 } \\
		& \defeq P_1 + P_2 + P_3 + P_4 \,.
	\end{align*}

	\step[L-FPE for $\bar m^{(2)}$]
	We may view $\bar m^{(2)}$ as also a solution to some L-FPE 
	\begin{equation*}
		\p_t \bar m^{(2)}(t;\mu,\delta_{z_1},\delta_{z_2}) = (\lin_{\mu_t}^\ast + \a_{\mu_t}^\ast) \bar m^{(2)}(t;\mu,\delta_{z_1},\delta_{z_2}) + \bar r(t;\mu,\delta_{z_1},\delta_{z_2}) \,,  
	\end{equation*}
	with $\bar r$ defined by 
	\begin{equation}
		\label{e:remainder-r-bar}
		\bar r(t) \defeq r(t) + q^{(1)}_{0,\infty} \cdot \paren[\Big]{ (\lin_{\mu_t}^\ast + \a_{\mu_t}^\ast) \grad m^{(1)}(t;\mu,\delta_{z_1}) - \grad (\lin_{\mu_t}^\ast + \a_{\mu_t}^\ast) m^{(1)}(t;\mu,\delta_{z_1}) } \,.
	\end{equation}
	Here 
	\begin{equation*}
		\ip{f, (\lin_{\mu_t}^\ast \grad - \grad \lin_{\mu_t}^\ast) q} = \ip{\lambda \grad f - \sigma^2 \phi \grad\phi \abs{\cdot-M(\mu_t)}^2 \lap f - \sigma^2 \phi^2 (\cdot-M(\mu_t)) \lap f, q }
	\end{equation*}
	and 
	\begin{align*}
		& \ip{f, (\a_{\mu_t}^\ast \grad - \grad \a_{\mu_t}^\ast) q} \\
		& = -\int \ip{ \lambda \grad f - \sigma^2 \phi^2 (\cdot-M(\mu_t)) \lap f, \mu_t } \frac{e^{-\alpha \efn(x)}}{\ip{e^{-\alpha\efn},\mu_t}} \\
		& \quad - \ip{\lambda \grad f - \sigma^2 \phi^2 (\cdot-M(\mu_t)) \lap f, \mu_t } \cdot (x-M(\mu_t)) \frac{-\alpha \grad \efn(x) e^{-\alpha \efn(x)}}{\ip{e^{-\alpha\efn},\mu_t}} \\
		& \quad + \left[ \ip{\lambda \grad \p_{x_j} f - \sigma^2 \phi^2 (\cdot-M(\mu_t)) \lap \p_{x_j} f, \mu_t} \cdot (x-M(\mu_t)) \frac{e^{-\alpha\efn(x)}}{\ip{e^{-\alpha\efn},\mu_t}} \right]_{j=1}^d q(dx) \,.
	\end{align*}
	Without ambiguity, we abbreviate $m^{(1)}(t;\mu,\delta_{z_i}) = q_i(t)$ for $i = 1,2$ and expand $q_2(t) = \tilde q_2(t) + q_\infty \cdot \grad \mu_t + q_\infty \cdot \grad (\delta_{\tilde x^\mu} - \mu_t)$. 
	Then the difference term in~\eqref{e:remainder-r-bar} can be written as 
	\begin{align}
		\nonumber 
		& q_\infty \cdot \paren[\Big]{ (\lin_{\mu_t}^\ast + \a_{\mu_t}^\ast) \grad q_1(t) - \grad (\lin_{\mu_t}^\ast + \a_{\mu_t}^\ast) q_1(t) } \\
		\label{e:r-diff-term-1}
		& = +\sum_{i=1}^d  \ip{\lambda q_\infty^i \p_{x_i} f - \sigma^2 \phi^2 \p_{x_i}^2 f q_\infty^\top (\cdot-M(\mu_t)), q_1(t)} \\
		& \qquad - \sigma^2 \sum_{i=1}^d \ip{\phi \p_{x_i}^2 f \abs{\cdot-M(\mu_t)}^2 q_\infty^\top \grad \phi, q_1(t)} \\
		& \quad - \left( \sum_{i=1}^d  \ip{\lambda q_\infty^i \p_{x_i} f - \sigma^2 \phi^2 \p_{x_i}^2 f q_\infty^\top (\cdot-M(\mu_t)),\mu_t} \right) \frac{\ip{e^{-\alpha\efn}, q_1(t)}}{\ip{e^{-\alpha\efn}, \mu_t}} \\
		& \quad - \sum_{i=1}^d \ip{\lambda \p_{x_i}f , \mu_t }  \frac{\ip{- q_\infty^\top \alpha \grad \efn G_t^i, q_1(t)}}{\ip{e^{-\alpha\efn},\mu_t}} \\
		& \quad + \sum_{i=1}^d \ip{\sigma^2 \phi^2 \p_{x_i}^2 f (\cdot-M(\mu_t)), \mu_t} \cdot \frac{\ip{- q_\infty^\top \alpha \grad \efn G_t, q_1(t)}}{\ip{e^{-\alpha\efn},\mu_t}}  \\
		\label{e:r-diff-term-2}
		& \quad + \sum_{i=1}^d q_\infty^i \ip{\lambda \grad \p_{x_i} f - \sigma^2 \phi^2 \lap \p_{x_i}^2 f (\cdot-M(\mu_t)), \mu_t} \cdot \frac{\ip{G_t, q_1(t)}}{\ip{e^{-\alpha \efn},\mu_t}} \,.
	\end{align}

	Recall from the remainder term $r$ that 
	\begin{equation*}
		P_1 = \ip{\lambda \grad f - \sigma^2 \phi^2 (\cdot-M(\mu_t)) \lap f, q_1(t)} \cdot \frac{\ip{G_t, q_2(t)}}{\ip{e^{-\alpha\efn}, \mu_t}} \,,
	\end{equation*}
	where 
	\begin{align*}
		\ip{G_t, q_2(t)} & = \ip{G_t, \tilde q_2(t)} - \sum_{i=1}^d q_\infty^i \left( \ip{\p_{x_i} G_t, \mu_t} + \ip{\p_{x_i} G_t, \delta_{\tilde x^\mu} - \mu_t} \right)  \\
		& = \ip{G_t, \tilde q_2(t)} - \sum_{i=1}^d q_\infty^i \left( \ip{e^{-\alpha\efn}, \mu_t} \ve_i - \ip{\alpha \p_{x_i} \efn G_t, \mu_t} + \ip{\p_{x_i} G_t, \delta_{\tilde x^\mu} - \mu_t} \right) \,.
	\end{align*}
	Then we may write $P_1$ as 
	\begin{align*}
		P_1 & = \ip{\lambda \grad f - \sigma^2 \phi^2 (\cdot-M(\mu_t)) \lap f, q_1(t)} \cdot \frac{\ip{G_t, \tilde q_2(t)}}{\ip{e^{-\alpha\efn}, \mu_t}} \\
		& \quad - \sum_{i=1}^d \ip{\lambda q_\infty^i \p_{x_i} f - \sigma^2 \phi^2 \lap f q_\infty^i (\cdot-M(\mu_t))_i, q_1(t) } \\
		& \quad + \ip{\lambda \grad f - \sigma^2 \phi^2 (\cdot-M(\mu_t)) \lap f, q_1(t)} \cdot \sum_{i=1}^d q_\infty^i (\ip{\p_{x_i} G_t, \delta_{\tilde x^\mu} - \mu_t} - \ip{\alpha \p_{x_i} \efn G_t, \mu_t}) \\
		& = \ip{\lambda \grad f - \sigma^2 \phi^2 (\cdot-M(\mu_t)) \lap f, q_1(t)} \cdot \frac{\ip{G_t, \tilde q_2(t)}}{\ip{e^{-\alpha\efn}, \mu_t}} \\
		& \quad - \sum_{i=1}^d \ip{\lambda q_\infty^i \p_{x_i} f - \sigma^2 \phi^2 \p_{x_i}^2 f q_\infty^\top (\cdot-M(\mu_t)), q_1(t) } \\
		& \quad + \ip{\lambda \grad f - \sigma^2 \phi^2 (\cdot-M(\mu_t)) \lap f, q_1(t)} \cdot \sum_{i=1}^d q_\infty^i (\ip{\p_{x_i} G_t, \delta_{\tilde x^\mu} - \mu_t} - \ip{\alpha \p_{x_i} \efn G_t, \mu_t}) \,.
	\end{align*}
	Note that the second last line cancels with~\eqref{e:r-diff-term-1}.

	Next, we study $P_2$, 
	\begin{equation*}
		\ip{\lambda \grad f - \sigma^2 \phi^2 (\cdot-M(\mu_t)) \lap f, q_2(t)} \cdot \frac{\ip{G_t, q_1(t)}}{\ip{e^{-\alpha\efn}, \mu_t}} \,.
	\end{equation*}
	Using the expansion $q_2(t) = \tilde q_2(t) + q_\infty \cdot \grad \mu_t + q_\infty \cdot \grad (\delta_{\tilde x^\mu} - \mu_t)$, we get 
	\begin{align*}
		P_2 & = \ip{\lambda \grad f - \sigma^2 \phi^2 (\cdot-M(\mu_t)) \lap f, \tilde q_2(t)} \cdot \frac{\ip{G_t, q_1(t)}}{\ip{e^{-\alpha\efn}, \mu_t}} \\
		& \quad - \sum_{i=1}^d q_\infty^i \ip{\lambda \grad \p_{x_i} f - \sigma^2 \phi^2 \lap \p_{x_i} f (\cdot-M(\mu_t)), \mu_t } \cdot \frac{\ip{G_t, q_1(t)}}{\ip{e^{-\alpha\efn}, \mu_t}} \\
		& \quad + \sum_{i=1}^d q_\infty^i \sigma^2 \ip{ 2 \phi \p_{x_i} \phi \lap f (\cdot-M(\mu_t)) + \phi^2 \lap f \ve_i, \mu_t } \cdot \frac{\ip{G_t, q_1(t)}}{\ip{e^{-\alpha\efn}, \mu_t}} \\
		& \quad + \sum_{i=1}^d q_\infty^i \ip{\lambda \grad f - \sigma^2 \phi^2 \lap f (\cdot-M(\mu_t)), \p_{x_i} (\delta_{\tilde x^\mu} - \mu_t)} \cdot \frac{\ip{G_t, q_1(t)}}{\ip{e^{-\alpha\efn}, \mu_t}}
	\end{align*}
	Notice that the second line cancels with~\eqref{e:r-diff-term-2}.

	Summarizing the above analysis, we have the remainder term $\bar r$ being 
	\begin{align*}
		\ip{f, \bar r(t)} & = \ip{\lambda \grad f - \sigma^2 \phi^2 (\cdot-M(\mu_t)) \lap f, q_1(t)} \cdot \frac{\ip{G_t, \tilde q_2(t)}}{\ip{e^{-\alpha\efn}, \mu_t}} \\
		& \qquad + \sigma^2 \sum_{i=1}^d \ip{\phi \p_{x_i}^2 f \abs{\cdot-M(\mu_t)}^2 q_\infty^\top \grad \phi, q_1(t)} \\
		& \quad - \ip{\lambda \grad f - \sigma^2 \phi^2 (\cdot-M(\mu_t)) \lap f, q_1(t)} \cdot \sum_{i=1}^d q_\infty^i (\ip{\p_{x_i} G_t, \delta_{\tilde x^\mu} - \mu_t} - \ip{\alpha \p_{x_i} \efn G_t, \mu_t})  \\
		& \qquad + \ip{\lambda \grad f - \sigma^2 \phi^2 (\cdot-M(\mu_t)) \lap f, \tilde q_2(t)} \cdot \frac{\ip{G_t, q_1(t)}}{\ip{e^{-\alpha\efn}, \mu_t}} \\
		& \quad + \sum_{i=1}^d q_\infty^i \sigma^2 \ip{ 2 \phi \p_{x_i} \phi \lap f (\cdot-M(\mu_t)) + \phi^2 \lap f \ve_i, \mu_t } \cdot \frac{\ip{G_t, q_1(t)}}{\ip{e^{-\alpha\efn}, \mu_t}} \\
		& \quad + \sum_{i=1}^d q_\infty^i \ip{\lambda \grad f - \sigma^2 \phi^2 \lap f (\cdot-M(\mu_t)), \p_{x_i} (\delta_{\tilde x^\mu} - \mu_t)} \cdot \frac{\ip{G_t, q_1(t)}}{\ip{e^{-\alpha\efn}, \mu_t}} \\
		& \quad - \ip {\lambda \grad f - \sigma^2 \phi^2 (\cdot-M(\mu_t)) \lap f, \mu_t} \cdot \\ 
		& \qquad \left( \frac{ \ip{G_t, q_1(t)} \ip{e^{-\alpha\efn}, q_2(t)} }{ \ip{e^{-\alpha\efn},\mu_t}^2 } +  \frac{ \ip{G_t, q_2(t)} \ip{e^{-\alpha\efn}, q_1(t)} }{ \ip{e^{-\alpha\efn},\mu_t}^2 } \right) \\
		& \quad + \ip{ \sigma^2 \phi^2 \lap f, \mu_t} \frac{ \ip{G_t, q_1(t)} \cdot \ip{G_t, q_2(t)} }{ \ip{e^{-\alpha\efn},\mu_t}^2 } \\
		& \quad - \left( \sum_{i=1}^d  \ip{\lambda q_\infty^i \p_{x_i} f - \sigma^2 \phi^2 \p_{x_i}^2 f q_\infty^\top (\cdot-M(\mu_t)),\mu_t} \right) \frac{\ip{e^{-\alpha\efn}, q_1(t)}}{\ip{e^{-\alpha\efn}, \mu_t}} \\
		& \quad - \sum_{i=1}^d \ip{\lambda \p_{x_i}f , \mu_t }  \frac{\ip{- q_\infty^\top \alpha \grad \efn G_t^i, q_1(t)}}{\ip{e^{-\alpha\efn},\mu_t}} \\
		& \quad + \sum_{i=1}^d \ip{\sigma^2 \phi^2 \p_{x_i}^2 f (\cdot-M(\mu_t)), \mu_t} \cdot \frac{\ip{- q_\infty^\top \alpha \grad \efn G_t, q_1(t)}}{\ip{e^{-\alpha\efn},\mu_t}}  \,.
	\end{align*}

	\step[Decay of remainder term]
	Let $g_i = \p_{x_i} f$ for $i \in [d]$. 
	Then we can write 
	\begin{equation*}
		\ip{f, \bar r(t)} = \sum_{i=1}^d \ip{g_i, \bar R^i(t)} \,,
	\end{equation*}
	where $R^i(t)$ is defined by 
	\begin{align}
		\nonumber
		& \ip{g_i, \bar R^i(t)} = \\
		\label{e:r-l1} \tag{$L_1$}
		& \qquad \ip{\lambda g_i \ve_i - \sigma^2 \phi^2 \p_{x_i} g_i (\cdot-M(\mu_t)), \mu_t} \cdot \frac{\ip{G_t, \tilde q_2(t)}}{\ip{e^{-\alpha\efn}, \mu_t}} \\
		\label{e:r-l2} \tag{$L_2$}
		& \quad + \ip{\sigma^2 \phi q_\infty^\top \grad \phi \abs{\cdot-M(\mu_t)}^2 \p_{x_i} g_i, q_1(t)} \\
		\label{e:r-l3} \tag{$L_3$}
		& \quad - \ip{\lambda g_i, q_1(t)} q_\infty^\top ( \ip{\grad G_t^i, \delta_{\tilde x^\mu} - \mu_t} + \ip{\alpha \grad \efn G_t^i, \mu_t} ) \\
		\label{e:r-l4} \tag{$L_4$}
		& \quad + \ip{\lambda g_i \ve_i - \sigma^2 \phi^2 \p_{x_i} g_i (\cdot-M(\mu_t)), \tilde q_2(t)} \cdot \frac{\ip{G_t, q_1(t)}}{\ip{e^{-\alpha\efn}, \mu_t}} \\
		\label{e:r-l5} \tag{$L_5$}
		& \quad + \ip{2\sigma^2 \phi q_\infty^\top \grad \phi (\cdot - M(\mu_t)) \p_{x_i} g_i, \mu_t} \cdot \frac{\ip{G_t, q_1(t)}}{\ip{e^{-\alpha\efn}, \mu_t}} \\
		\label{e:r-l6} \tag{$L_6$}
		& \quad + \ip{\sigma^2 \phi^2 \p_{x_i} g_i, \mu_t} \frac{q_\infty^\top \ip{G_t, q_1(t)}}{\ip{e^{-\alpha\efn}, \mu_t}} \\
		\label{e:r-l7} \tag{$L_7$}
		& \quad + \ip{ \lambda g_i \ve_i - \sigma^2 \phi^2 (\cdot-M(\mu_t)) \p_{x_i} g_i, q_\infty^\top \grad(\delta_{\tilde x^\mu}-\mu_t) } \cdot \frac{\ip{G_t^i, q_1(t)}}{\ip{e^{-\alpha\efn},\mu_t}} \\
		\label{e:r-l8} \tag{$L_8$}
		& \quad - \ip {\lambda g_i \ve_i - \sigma^2 \phi^2 (\cdot-M(\mu_t)) \p_{x_i} g_i, \mu_t} \cdot \\ 
		\nonumber 
		& \qquad \left( \frac{ \ip{G_t, q_1(t)} \ip{e^{-\alpha\efn}, q_2(t)} }{ \ip{e^{-\alpha\efn},\mu_t}^2 } +  \frac{ \ip{G_t, q_2(t)} \ip{e^{-\alpha\efn}, q_1(t)} }{ \ip{e^{-\alpha\efn},\mu_t}^2 } \right) \\
		\label{e:r-l9} \tag{$L_9$}
		& \quad + \ip{ \sigma^2 \phi^2 \p_{x_i} g_i, \mu_t} \frac{ \ip{G_t, q_1(t)} \cdot \ip{G_t, q_2(t)} }{ \ip{e^{-\alpha\efn},\mu_t}^2 } \\
		\label{e:r-l10} \tag{$L_{10}$}
		& \quad - \ip{\lambda q_\infty^i g_i - \sigma^2 \phi^2 \p_{x_i} g_i q_\infty^\top (\cdot-M(\mu_t)),\mu_t}  \frac{\ip{e^{-\alpha\efn}, q_1(t)}}{\ip{e^{-\alpha\efn}, \mu_t}} \\
		\label{e:r-l11} \tag{$L_{11}$}
		& \quad - \ip{\lambda g_i , \mu_t }  \frac{\ip{- q_\infty^\top \alpha \grad \efn G_t^i, q_1(t)}}{\ip{e^{-\alpha\efn},\mu_t}} \\
		\label{e:r-l12} \tag{$L_{12}$}
		& \quad + \ip{\sigma^2 \phi^2 \p_{x_i} g_i (\cdot-M(\mu_t)), \mu_t} \cdot \frac{\ip{- q_\infty^\top \alpha \grad \efn G_t, q_1(t)}}{\ip{e^{-\alpha\efn},\mu_t}}  \,.
	\end{align}
	We will show that $\norm{R^i(t)}_{(5,\infty)'} \le K e^{-\kappa_0 t}$ for all $i \in [d]$ and $t \ge 0$. 

	Fix $i$.
	Recall that 
	\begin{equation*}
		\norm{\tilde q_2(t)}_{(4,\infty)'} \le C e^{-\kappa_0 t} \,,
	\end{equation*}
	so for~\eqref{e:r-l1} and~\eqref{e:r-l4} we have 
	\begin{equation*}
		\abs{L_1} \le C_1(\lambda,\sigma,r_\cut) \norm{g_i}_{1,\infty} \frac{\norm{G_t}_{(4,\infty)} \norm{\tilde q_2(t)}_{(4,\infty)'}}{\ip{e^{-\alpha\efn}, \mu_t}} \le C_1' e^{-\kappa_0 t} \norm{g_i}_{1,\infty}
	\end{equation*}
	and similarly 
	\begin{equation*}
		\abs{L_4} \le C_4(\lambda,\sigma,r_\cut) \norm{g_i}_{5,\infty} \norm{\tilde q_2(t)}_{(4,\infty)'} \frac{\norm{G_t}_{(4,\infty)} \norm{q_2(t)}_{(4,\infty)'}}{\ip{e^{-\alpha\efn}, \mu_t}} \le C_4' e^{-\kappa_0 t} \norm{g_i}_{5,\infty} \,,
	\end{equation*}
	due to the fact that $G_t$ has bounded derivatives. 

	Recall that $G_t(x) = (x - M(\mu_t)) e^{-\alpha \efn(x)}$. 
    Using the fact that $\norm{\delta_0 - \mu_t}_{(1,\infty)'} \le C e^{-\kappa t}$ and
	\begin{equation*}
		\int \abs{x - M(\mu_t)} \mu_t(dx) \le C e^{-\kappa t} \,,
	\end{equation*}
	We see the exponential decay in items of the form $\ip{h, \delta_{\tilde x^\mu} - \mu_t}$, $\ip{(\cdot - M(\mu_t)) h, \mu_t}$, and $\ip{(\cdot - M(\mu_t)) h, \delta_{\tilde x^\mu}}$.
    Also, $\abs{q_\infty}$ is uniformly bounded by Remark~\ref{rm:q-infty-bound}.
    Then, for~\eqref{e:r-l3}, \eqref{e:r-l5}, \eqref{e:r-l7}, and~\eqref{e:r-l12}, we have
	\begin{equation*}
		\abs{L_3 + L_5 + L_7 + L_{12}} \le C_3 e^{-\kappa t} \norm{g_i}_{(5,\infty)} \,. 
	\end{equation*}

	Now, recall that $q_1(t) = m^{(1)}(t;\mu,\delta_{z_1}) = \tilde m^{(1)}(t;\mu,\delta_{z_1}) + q^{(1)}_{0,\infty}(\mu,z_1) \cdot \grad \delta_{\tilde x^\mu}$, 
	and we abbreviate as $q_1(t) = \tilde q_1(t) + q_{\infty,1} \cdot \grad \delta_{\tilde x^\mu}$. 
	For~\eqref{e:r-l2} we have 
	\begin{equation*}
		L_2 = \ip{\sigma^2 \phi q_\infty^\top \grad \phi \abs{\cdot-M(\mu_t)}^2 \p_{x_i} g_i, \tilde q_1(t)} - \ip{\sigma^2 q_{\infty,1}^\top \grad (\phi q_\infty^\top \grad \phi \abs{\cdot-M(\mu_t)}^2 \p_{x_i} g_i), \delta_{\tilde x^\mu}} \,,
	\end{equation*}
	so 
	\begin{equation*}
		\abs{L_2} \le C_2 (e^{-\kappa_0 t} + e^{-\kappa t}) \norm{g_i}_{(5,\infty)} \,.
	\end{equation*}

	Notice that~\eqref{e:r-l6} itself is asymptotically non-decaying.
	But using the decomposition $q_2(t) = \tilde q(2) + q_\infty^\top \grad \mu_t + q_\infty^\top \grad (\delta_0 - \mu_t)$, \eqref{e:r-l9} becomes
	\begin{align*}
		\sum_{j=1}^d \ip{\sigma^2 \phi^2 \p_{x_i} g_i, \mu_t} \frac{\ip{G_t^j, q_1(t)}}{\ip{e^{-\alpha\efn}, \mu_t}}  \left( \frac{\ip{G_t^j, \tilde q_2(t)}}{\ip{e^{-\alpha\efn}, \mu_t}} - q_\infty^j + \frac{\ip{q_\infty^\top \alpha \grad \efn G_t^j, \mu_t}}{\ip{e^{-\alpha\efn}, \mu_t}} - \frac{\ip{q_\infty^\top \grad G_t^j, \delta_0-\mu_t}}{\ip{e^{-\alpha\efn}, \mu_t}} \right) \,,
	\end{align*}
	where the $q_\infty^j$-term cancels with~\eqref{e:r-l6}. 
	Then 
	\begin{equation*}
		\abs{L_6 + L_9} \le C_6 (e^{-\kappa_0 t} + e^{-\kappa t}) \norm{g_i}_{1,\infty} \,.
	\end{equation*}

    Finally, we study~\eqref{e:r-l8} and~\eqref{e:r-l10}.
    Observe that the $\sigma^2$-terms contain the form $\ip{(\cdot-M(\mu_t)) h, \mu_t}$, which we know admits exponential decay from the above analysis. 
    It remains to deal with 
    \begin{align*}
        -\lambda \ip{g_i, \mu_t} \frac{ \ip{G_t^i, q_1(t)} \ip{e^{-\alpha\efn}, q_2(t)} + (\ip{G_t^i, q_2(t)} + q_{\infty}^i \ip{e^{-\alpha\efn}, \mu_t}) \ip{e^{-\alpha\efn}, q_1(t)} }{\ip{e^{-\alpha\efn}, \mu_t}^2} \,.
    \end{align*}
    Notice that 
    \begin{align*}
        \ip{G_t^i, q_2(t)} & = \ip{G_t^i, \tilde q_2(t)} - q_\infty^\top \sum_{j=1}^d \p_{x_j} G_t^i(\tilde x^\mu) \\
        & = \ip{G_t^i, \tilde q_2(t)} - q_{\infty}^i e^{-\alpha\efn(\tilde x^\mu)} + \sum_{j=1}^d \alpha q_\infty^j \p_{x_j} \efn(\tilde x^\mu) (\tilde x^\mu_i - M(\mu_t)_i) \,,
    \end{align*}
    so $\ip{G_t^i, q_2(t)} + q_{\infty}^i \ip{e^{-\alpha\efn}, \mu_t}$ admits exponential decay of rate $\kappa \land \kappa_0$. 
    On the other hand, 
    \begin{align*}
        \ip{e^{-\alpha\efn}, q_2(t)} & = \ip{e^{-\alpha\efn}, \tilde q_2(t)} + \alpha q_{\infty}^\top \grad \efn(\tilde x^\mu) e^{-\alpha \efn(\tilde x^\mu)} \,.
    \end{align*}
    We combine the second term with~\eqref{e:r-l11} to see that 
    \begin{align*}
        & \frac{\alpha q_\infty^\top \grad \efn(\tilde x^\mu) e^{-\alpha\efn(\tilde x^\mu)} \ip{G_t^i, q_1(t}}{\ip{e^{-\alpha\efn}, \mu_t}} - \alpha q_\infty^\top \ip{\grad \efn G_t^i, q_1(t)} \\
        & = \sum_{j=1}^d \alpha q_\infty^j \ip{ \paren[\big]{ \frac{\p_{x_j} \efn(\tilde x^\mu) e^{-\alpha \efn(\tilde x^\mu)}}{\ip{e^{-\alpha\efn}, \mu_t}} - \p_{x_j} \efn} G_t^i, q_1(t) } \\
        & = \sum_{j=1}^d \alpha q_\infty^j \ip{ \paren[\big]{ \frac{\p_{x_j} \efn(\tilde x^\mu) e^{-\alpha \efn(\tilde x^\mu)}}{\ip{e^{-\alpha\efn}, \mu_t}} - \p_{x_j} \efn} G_t^i, \tilde q_1(t) } \\
        & \qquad + \sum_{j=1}^d \alpha q_\infty^j \sum_{\ell=1}^d q_{\infty,1}^\top \ip{ \paren[\big]{ \frac{\p_{x_j} \efn(\tilde x^\mu) e^{-\alpha \efn(\tilde x^\mu)}}{\ip{e^{-\alpha\efn}, \mu_t}} - \p_{x_j} \efn} G_t^i, \p_{x_\ell} \delta_{\tilde x^\mu}} \,.
    \end{align*}
    Note that 
    \begin{align*}
        \p_{x_\ell} \paren[\big]{ \frac{\p_{x_j} \efn(\tilde x^\mu) e^{-\alpha \efn(\tilde x^\mu)}}{\ip{e^{-\alpha\efn}, \mu_t}} - \p_{x_j} \efn} \vert_{x=\tilde x^\mu} & = - \p_{x_\ell} \p_{x_j} \efn(\tilde x^\mu) G_t^i(\tilde x^\mu) \\
        & \qquad + \p_{x_j} \efn(\tilde x^\mu) \paren[\big]{ \frac{e^{-\alpha\efn(\tilde x^\mu)}}{\ip{e^{-\alpha\efn}, \mu_t}} - 1 } \p_{\ell} G_t^i (\tilde x^\mu) \,,
    \end{align*}
    which displays exponential decay again due to Proposition~\ref{pp:CBO-mf-decay}.
    Joining the above analysis on~\eqref{e:r-l8} and~\eqref{e:r-l10}, we get 
    \begin{equation*}
        \abs{L_8 + L_{10} + L_{11}} \le C_8 e^{-(\kappa \land \kappa_0) t} \norm{g_i}_{1,\infty} \,.
    \end{equation*}

	To sum up, we see that $\abs{\ip{g_i, R^i_t}} \le K \norm{g_i}_{(5,\infty)'} e^{-\kappa_0 t}$, where $\kappa_0 \le \kappa$, so 
	\begin{equation*}
		\norm{R^i_t}_{(5,\infty)'} \le K e^{-\kappa_0 t} \,.
	\end{equation*}
	That verifies the pre-condition of Lemma~\ref{lm:bound-q} on $r$. 
	Notice that 
	\begin{equation*}
		\bar m^{(2)}(0;\mu,\delta_{z_1},\delta_{z_2}) = (\mu_0 - \delta_{z_2}) - q^{(1)}_{0,\infty}(\mu,z_2) \cdot \grad (\delta_{z_1} - \mu_0)) \,.
	\end{equation*}
	A slight generalization of Lemma~\ref{lm:bound-q} gives the overall bound 
	\begin{equation*}
		\norm{\bar m^{(2)}(t;\mu,\delta_{z_1},\delta_{z_2})}_{(6,\infty)'} \le C(1+K) \,.
	\end{equation*}
\end{proof}

\begin{proof}[Substitute proof of Lemma~\ref{l:d2-bound}]
	By replacing all $m^{(1)}(t;\mu,\delta_z)$ with $d^{(1)}_j(t;\mu,z)$ in the above proof, we get exactly~\eqref{e:d2-r-decay} for Lemma~\ref{l:d2-bound}.

	For~\eqref{e:d2-r-bounded}, the expansion $q(t) = \tilde q(t) + q_\infty \cdot \grad \delta_{\tilde x^\mu}$ is unnecessary. 
	We directly apply the uniform-in-time upper bounds (first inequality) in Lemma~\ref{l:d1-bound} to see that 
	\begin{equation*}
		\abs{\ip{g, \bar R^i_t}} \le C \norm{g}_{3,\infty}
	\end{equation*}
	for every $i$. 

	Now, recall that 
	\begin{align*}
		\norm{q_t}_{(4,\infty)'} & = \sup_{\norm{\xi}_{4,\infty} \le 1} \ip{\xi, q_t} \\
		& = \sup_{\norm{\xi}_{4,\infty} \le 1} \ip{w(0), q_0} + \int_0^t \sum_{i=1}^d \ip{\p_{x_i} w(s), R^i_s} ds \,,
	\end{align*}
	where $w: [0,t] \times \R^d \to \R$ is the solution to 
	\begin{equation*}
		\begin{cases}
			\p_s w(s) + \lin_{\mu_s} w(s) + \a_{\mu_s} w(s) = 0 \,, & s \in [0,t] \,, \\
			w(t) = \xi \,.
		\end{cases}
	\end{equation*}
	We apply~\eqref{e:d2-r-bounded} and Lemma~\ref{lm:prop-y} to see that 
	\begin{equation*}
		\norm{\bar d^{(2)}_{j,j}(t;\mu,z_1,z_2)}_{(4,\infty)'} \le C t \,.
	\end{equation*}
	Then, plug this into the conclusion of Lemma~\ref{lm:connect-r1-q}, we see that the pre-condition of Lemma~\ref{lm:erg-at-delta} is satisfied but with 
	\begin{equation*}
		\norm{R^i_t}_{(5,\infty)'} \le K (t e^{-\kappa t} + e^{-\kappa_0 t}) \,, \qquad t \ge 0 \,,
	\end{equation*}
	which is still asymptotically exponential.
	Thus we conclude that 
	\begin{equation*}
		\norm{\tilde d^{(2)}_{j,j}(t;\mu,z_1,z_2)}_{(6,\infty)'} \le C e^{-\kappa_0' t} 
	\end{equation*}
	for all $t \ge 0$, but with some $\kappa_0' >0$ possibly different from $\kappa_0$ coming from Lemma~\ref{l:d1-bound}.
\end{proof}

\section{Technical lemmata}
\label{s:technical}

\begin{lemma}
	\label{lm:prop-gbm}
	Consider the following generic stochastic differential equation on $\R^d$,
	\begin{equation*}
		dS_u = -\lambda S_u du + \sigma \abs{S_u} \phi(S_u) dW_u \,.
	\end{equation*}
	We denote by $(S^{t,x}_u)_{u \in [t,T]}$ the process that satisfies the above SDE on time interval $[t,T]$ with initial data $S_t = x$. 
	When $\lambda$ is large enough, for any $i \in [d]$, $p \in \N$, and nonzero multi-exponent $\beta \in \N^d$ such that $p + \abs{\beta} \le 5$, we have 
	\begin{equation*}
		\E \left[ \abs{\p_x^\beta (\p_{x_i} S^{t,x}_u)} \abs{\p_{x_i} S^{t,x}_u}^p \right] \le e^{-\frac{\lambda}{2} (u-t)} \,, \qquad u \in [t,T] \,.
	\end{equation*}
\end{lemma}

\begin{proof}
	We first look at $\p_{x_j} S^{t,x}_u$ for an arbitrary $j \in [d]$. 
	Observe that, for $u \in [t,T]$,
	\begin{equation*}
		d \p_{x_j} S^{t,x}_u = -\lambda \p_{x_j} S^{t,x}_u du + \sigma \p_{x_j} S^{t,x}_u \cdot \paren[\big]{ \frac{S^{t,x}_u \phi(S^{t,x}_u)}{\abs{S^{t,x}_u}} + \abs{S^{t,x}_u} \grad\phi(S^{t,x}_u) } dW_u  \,,
	\end{equation*}
	where with $c_1$ defined for $\phi$ we have 
	\begin{equation*}
		\abs{ \frac{S^{t,x}_u \phi(S^{t,x}_u)}{\abs{S^{t,x}_u}} + \abs{S^{t,x}_u} \grad\phi(S^{t,x}_u) } \le 1 + c_1 d \,, \qquad \P\text{-a.s.,}
	\end{equation*}
	Applying It\^{o}'s formula to $\abs{\p_{x_j} S^{t,x}_u}^p$ with $2 \le p \le 10$, we see that 
	\begin{equation*}
		\frac{d}{du} \E \abs{\p_{x_j} S^{t,x}_u}^p \le -(p\lambda - \frac{p(p-1)}{2} \sigma^2 (1+c_1 d)) \E \abs{\p_{x_j} S^{t,x}_u}^p \,.
	\end{equation*}
	When $\lambda$ is large enough, $2\lambda - (p-1) \sigma^2 (1+c_1 d) > \lambda$, which proves for $\abs{\beta} = 0$ and $p \ge 2$. 
	Also, Jensen's inequality gives the conclusion for $p=1$ as well. 

	On the other hand, for multi-index $\beta$, we have 
	\begin{equation*}
		d \p_x^\beta (\p_{x_j} S^{t,x}_u) = -\lambda \p_x^\beta (\p_{x_j} S^{t,x}_u) du + \sigma \p_x^\beta \left[ \p_{x_j} S^{t,x}_u \cdot \paren[\big]{ \frac{S^{t,x}_u \phi(S^{t,x}_u)}{\abs{S^{t,x}_u}} + \abs{S^{t,x}_u} \grad\phi(S^{t,x}_u) } \right] dW_u  \,.
	\end{equation*}
	Similarly, with the $W^{6,\infty}$-boundedness of $\phi$, we have 
	\begin{equation*}
		\frac{d}{du} \E \abs{\p_x^\beta (\p_{x_j} S^{t,x}_u)}^2 \le -(2\lambda - \sigma^2 \kappa_{2,\beta}) \E \abs{\p_x^\beta (\p_{x_j} S^{t,x}_u)}^2
	\end{equation*}
	for some $\kappa_{2,\beta}$ depending only on $\phi$ and $\beta$. 
	We choose $\lambda > \sigma^2 \kappa_{2,\beta}$ and apply Jensen's inequality.
	This proves for any $\beta$ and $p=0$. 

	Finally, Cauchy-Schwarz inequality leads to the most general form. 
\end{proof}

\begin{lemma}
	\label{lm:prop-y}
	For $0 \le t < T$, $x \in \R^d$, $\mu \in \c_\cbo$, we define the process $(Y^{t,x,\mu}_u)_{u \in [t,T]}$ by 
	\begin{equation*}
		dY^{t,x,\mu}_u = -\lambda (Y^{t,x,\mu}_u - M(\mu_u)) du + \sigma \abs{Y^{t,x,\mu}_u - M(\mu_u)} dW_u \,, \quad Y^{t,x,\mu}_t = x \,.
	\end{equation*}
	Then, for any $i,k \in [d]$, $u \in [t,T]$, we have 
	\begin{equation*}
		\E [\p_{x_i} Y^{t,x,\mu}_{u,k}] \le e^{-\lambda (u-t)} \one{i=k} \,,
	\end{equation*}
	and 
	\begin{equation*}
		\E [\p_{x_i}^2 Y^{t,x,\mu}_{u,k}] = 0 \,.
	\end{equation*}
	Moreover, when $\lambda$ is large enough, there exist some $C_p, \kappa_p > 0$ (depending on $d$ and $\phi$), for $p=2,4$, such that
	\begin{equation*}
		\E [\abs{\p_{x_i} Y^{t,x,\mu}_{u}}^p] \le C_p e^{-(p\lambda - \sigma^2 \kappa_p)(u-t)} 
	\end{equation*}
	for all $i,k \in [d]$, $u \in [t,T]$.
\end{lemma}

\begin{proof}
	Fix $t,T,x,\mu$, and for simplicity we write $Y = Y^{t,x,\mu}$.

	Observe that $M(\mu_u)$ is independent of $x$, so we apply the techniques in the proof of Lemma~\ref{lm:prop-gbm} to see that 
	\begin{equation*}
		d \p_{x_i} Y_u = -\lambda \p_{x_i} Y_u du + \sigma \p_{x_i} Y_u \cdot \paren[\big]{ \frac{(Y_u - M(\mu_u)) \phi(Y_u)}{\abs{Y_u - M(\mu_u)}} + \abs{Y_u - M(\mu_u)} \grad\phi(Y_u) } dW_u  \,.
	\end{equation*}
	Then 
	\begin{equation*}
		\frac{d}{du} \E [\p_{x_i} Y_u] = -\lambda \E[\p_{x_i} Y_u] \,.
	\end{equation*}
	Note that $\E[Y_t] = x$, and $\p_{x_i} x = \one{i=k}$. 
	So 
	\begin{equation*}
		\E [\p_{x_i} Y_{u,k}] = e^{-\lambda (u-t)} \one{i=k} \,.
	\end{equation*}
	Similarly, taking a further derivative gives also 
	\begin{equation*}
		\frac{d}{du} \E [\p_{x_i}^2 Y_u] = -\lambda \E[\p_{x_i}^2 Y_u] \,,
	\end{equation*}
	while $\p_{x_i}^2 x = 0$. 
	Thus 
	\begin{equation*}
		\E [\p_{x_i}^2 Y_{u,k}] = 0 \,.
	\end{equation*}

	Recall that  
	\begin{align*}
		& d \p_{x_i} Y_{u,k} = -\lambda \p_{x_i} Y_{u,k} du \\
		& \qquad + \sigma \left[ \p_{x_i} Y_u \cdot \paren[\big]{ \frac{(Y_u - M(\mu_u)) \phi(Y_u)}{\abs{Y_u - M(\mu_u)}} + \abs{Y_u - M(\mu_u)} \grad\phi(Y_u) } \right] dW_{u,k}  \,.
	\end{align*}
	For $p \ge \{2,4\}$, we apply It\^{o}'s formula to see that 
	\begin{align*}
		& d (\p_{x_i} Y_{u,k})^p = -p\lambda (\p_{x_i} Y_{u,k})^p du \\
		& \qquad + \frac{p(p-1) \sigma^2}{2} (\p_{x_i} Y_{u,k})^{p-2} \left[ \p_{x_i} Y_u \cdot \paren[\big]{ \frac{(Y_u - M(\mu_u)) \phi(Y_u)}{\abs{Y_u - M(\mu_u)}} + \abs{Y_u - M(\mu_u)} \grad\phi(Y_u) } \right] du \\
		& \qquad + p\sigma (\p_{x_i} Y_{u,k})^{p-1} \left[ \p_{x_i} Y_u \cdot \paren[\big]{ \frac{(Y_u - M(\mu_u)) \phi(Y_u)}{\abs{Y_u - M(\mu_u)}} + \abs{Y_u - M(\mu_u)} \grad\phi(Y_u) } \right] dW_{u,k} \,.
	\end{align*}
	Notice that 
	\begin{equation*}
		{ \frac{(Y_u - M(\mu_u)) \phi(Y_u)}{\abs{Y_u - M(\mu_u)}} + \abs{Y_u - M(\mu_u)} \grad\phi(Y_u) }
	\end{equation*}
	is bounded almost surely, so there exists some constant $c_p > 0$, depending solely on $\phi$, such that 
	\begin{equation}
		\label{e:ineq-cp}
		\frac{d}{du} \E [(\p_{x_i} Y_{u,k})^p] \le -p\lambda \E[(\p_{x_i} Y_{u,k})^p] + \frac{p(p-1)\sigma^2 c_p}{2} \E [(\p_{x_i} Y_{u,k})^{p-2} \abs{\p_{x_i} Y_u}^2] \,.
	\end{equation}
	In particular, 
	\begin{equation*}
		\frac{d}{du} \E [(\p_{x_i} Y_{u,k})^2] \le -2\lambda \E[(\p_{x_i} Y_{u,k})^2] + \sigma^2 c_p\E \abs{\p_{x_i} Y_u}^2 \,,
	\end{equation*}
	so that 
	\begin{equation*}
		\frac{d}{du} \E \abs{\p_{x_i} Y_u}^2 \le -(2\lambda - \sigma^2 c_2) \E \abs{\p_{x_i} Y_u}^2 \,.
	\end{equation*}
	Note that $\E \abs{\p_{x_i} Y_t}^2 = 1$. 
    We may choose $C_2 = 1$, $\kappa_2 = c_2$ so that 
	\begin{equation*}
		\E (\p_{x_i} Y_{u,k})^2 \le \E \abs{\p_{x_i} Y_u}^2 \le C_2 e^{-(2\lambda - \sigma^2 \kappa_2) (u-t)} \,. 
	\end{equation*}

	Moreover, from~\eqref{e:ineq-cp} we have 
	\begin{equation*}
		\frac{d}{du} \E [(\p_{x_i} Y_{u,k})^4] \le -p\lambda \E[(\p_{x_i} Y_{u,k})^4] + 6 \sigma^2 c_4 \E [(\p_{x_i} Y_{u,k})^2 \abs{\p_{x_i} Y_u}^2] \,,
	\end{equation*}
	so 
	\begin{align*}
		\frac{d}{du} \E [\abs{\p_{x_i} Y_{u,k}}_4^4] & \le -p\lambda \E[\abs{\p_{x_i} Y_{u,k}}_4^4] + 6 \sigma^2 c_4 \E [\abs{\p_{x_i} Y_u}_2^4] \\
		& \le  -p\lambda \E[\abs{\p_{x_i} Y_{u,k}}_4^4] + 6 \sigma^2 c_4 d \E [\abs{\p_{x_i} Y_u}_4^4] \,.
	\end{align*}
	This gives 
	\begin{equation*}
		\E [\abs{\p_{x_i} Y_{u,k}}_4^4] \le e^{-(4\lambda - 6\sigma^2 c_4 d) (u-t)} \,.
	\end{equation*}
	Taking $C_4 = d$ and $\kappa_4 = 6dc_4$ proofs the claim for $p=4$.
\end{proof}

\end{document}